\documentclass[final,onefignum,onetabnum]{siamart171218}
\usepackage{booktabs}
\usepackage{multirow}

\usepackage{lipsum}
\usepackage{amsfonts}
\usepackage{graphicx}
\usepackage{epstopdf}
\usepackage{algorithmic}
\usepackage{amsmath}
\ifpdf
  \DeclareGraphicsExtensions{.eps,.pdf,.png,.jpg}
\else
  \DeclareGraphicsExtensions{.eps}
\fi

\newsiamremark{remark}{Remark}
\newsiamremark{hypothesis}{Hypothesis}
\newsiamremark{exmp}{Example}
\crefname{section}{Section}{Sections}
\crefname{subsection}{Subsection}{Subsections}
\crefname{hypothesis}{Hypothesis}{Hypotheses}
\newsiamthm{claim}{Claim}

\headers{Constraint-preserving OEDG for Einstein--Euler system}{Yuchen Huang, Manting Peng, and Kailiang Wu}

\title{Constraint-Preserving High-Order Compact OEDG Method for Spherically Symmetric Einstein--Euler System\thanks{This work was partially supported by Science Challenge Project (No.~TZ2025007), the Shenzhen Science and Technology Program (Grant No.~RCJC20221008092757098), and the National Natural Science Foundation of China (Grant No.~124B2022).}}

\author{Yuchen Huang\thanks{Department of Mathematics, Southern University of Science and Technology, Shenzhen, Guangdong 518055, China. 
  (\email{12432019@mail.sustech.edu.cn}).}
  \and Manting Peng\thanks{Department of Mathematics, Southern University of Science and Technology, Shenzhen, Guangdong 518055, China. 
  (\email{pengmt2024@mail.sustech.edu.cn}).}
\and Kailiang Wu\thanks{Corresponding author. Department of Mathematics and Shenzhen International Center for Mathematics, Southern University of Science and Technology, Shenzhen, Guangdong 518055, China
  (\email{wukl@sustech.edu.cn}).}
}

\usepackage{amsopn}

\makeatletter
\newcommand*{\addFileDependency}[1]{
  \typeout{(#1)}
  \@addtofilelist{#1}
  \IfFileExists{#1}{}{\typeout{No file #1.}}
}
\makeatother

\newcommand*{\myexternaldocument}[1]{%
    \externaldocument{#1}%
    \addFileDependency{#1.tex}%
    \addFileDependency{#1.aux}%
}

\ifpdf
\hypersetup{
	pdftitle={OEDG Method for the EE System},
	pdfauthor={Y. Huang, M. Peng, and K. Wu}
}
\fi

\myexternaldocument{ex_supplement}

\usepackage{enumitem}
\newlist{steps}{enumerate}{1}
\setlist[steps, 1]{label =\textbf{ Step \arabic*:}}

\newtheorem{example}{Example}
\begin{document} 
	
	\maketitle

\begin{abstract}
	Numerical simulation of the spherically symmetric Einstein--Euler (EE) system faces severe challenges due to the stringent physical admissibility constraints of relativistic fluids and the geometric singularities inherent in metric evolution. This paper proposes a high-order Constraint-Preserving (CP) compact Oscillation-Eliminating Discontinuous Galerkin (cOEDG) method specifically tailored to address these difficulties. The method integrates a scale-invariant oscillation-eliminating mechanism [M. Peng, Z. Sun, K. Wu, {\em Math. Comp.}, 94: 1147--1198, 2025] into a compact Runge--Kutta DG framework. By characterizing the convex invariant region of the hydrodynamic subsystem with general barotropic equations of state, we prove that the proposed scheme preserves physical realizability (specifically, positive density and subluminal velocity) directly in terms of conservative variables, thereby eliminating the need for complex primitive-variable checks. To ensure the geometric validity of the spacetime, we introduce a bijective transformation of the metric potentials. Rather than evolving the constrained metric components directly, the scheme advances unconstrained auxiliary variables whose inverse mapping automatically enforces strict positivity and asymptotic bounds without any limiters. Combined with a compatible high-order boundary treatment, the resulting CPcOEDG method exhibits robust stability and design-order accuracy in capturing strong gravity-fluid interactions, as demonstrated by simulations of black hole accretion and relativistic shock waves.
\end{abstract}



	\section{Introduction}
	Einstein's theory of general relativity \cite{einstein1916foundation} asserts that matter and energy curve spacetime. In regimes dominated by strong relativistic effects, such as black holes, neutron stars, and high-energy transients, Newtonian approximations are no longer adequate, and fluid dynamics must be reformulated on a curved spacetime governed by the Einstein equations. The canonical model for this interaction is the {Einstein--Euler (EE)} system, where the relativistic Euler equations for the fluid are tightly coupled with the dynamical evolution of the spacetime metric.
	
	The EE system is highly nonlinear, and exact solutions are rarely available in physically interesting settings. A central numerical difficulty is that smooth initial data typically evolve into discontinuous solutions (e.g., shocks), causing high-order schemes to develop spurious Gibbs-type oscillations. For the EE system, these oscillations are often spurious, and they tend to violate essential physical constraints (e.g., positivity of density and pressure, subluminal velocity), which can lead to nonlinear instabilities and eventual blow-up of the simulation. Hence, robust, high-resolution, {constraint-preserving} (CP) numerical schemes are indispensable for accurately predicting relativistic dynamics and interpreting multimessenger observations.

	The discontinuous Galerkin (DG) method, introduced by Reed and Hill \cite{Reed1973} and systematically developed by Cockburn and Shu \cite{CockburnHouShu1990,CockburnLinShu1989a,CockburnShu1989b,CockburnShu1991,CockburnShu1998}, provides a powerful framework for such simulations thanks to its high-order accuracy, inherent parallelism, and geometric flexibility. However, applying DG methods to the EE system raises two major challenges:
	
	(1) \textbf{Constraint preservation.}	
	Physical constraints on both the spacetime metric and the hydrodynamic state are crucial for physical fidelity and numerical stability. These include positivity of density and pressure, subluminal velocity, and admissibility of the metric components. Violating these constraints may generate complex eigenvalues of the Jacobian, rendering the discrete problem ill-posed and causing simulation failure. For hyperbolic conservation laws, two main approaches are widely used to enforce such bounds for high-order schemes: flux-correction limiters \cite{kuzmin2022bound,xu2014parametrized,WuTang2015}, which blend high-order fluxes with first-order bound-preserving ones, and local scaling limiters \cite{zhang2010positivity,zhang2010maximum,WuTang2017,Wu2021}, which guarantee pointwise constraints provided that the cell averages are admissible. 
	
	(2) \textbf{Control of spurious oscillations and enhancement of robustness.}  
	Typical oscillation-control strategies include slope limiters~\cite{kuzmin2004high,zhong2013simple,qiu2005runge,Dumbser2016JCP}, artificial viscosity~\cite{huang2020adaptive,zingan2013implementation,yu2020study}, and damping-based frameworks. Recently, Liu, Lu, and Shu~\cite{LuLiuShu2021,LiuLuShu2022} proposed an innovative damping-based technique, the oscillation-free DG (OFDG) method, which maintains $L^2$ stability, consistency, and optimal error estimates.  However, the damping term in OFDG is not scale-invariant. In the presence of strong shocks, the damping can become stiff and severely restrict the allowable time step and efficiency unless one resorts to exponential Runge--Kutta methods \cite{LiuLuShu2022}.
	
	To overcome this limitation, the oscillation-eliminating DG (OEDG) method was developed in \cite{PengSunWu2025}. In OEDG, the original damping is replaced by a separate, scale-invariant {oscillation-eliminating (OE)} procedure that acts as a filter. OE procedure, applied after each Runge--Kutta (RK) stage, exactly solves an auxiliary ODE to modify high-order modes using scale-invariant jump information. The method preserves conservation, cell averages, and optimal convergence, maintains stability under standard CFL conditions, and can be nonintrusively appended to existing DG codes. It has been successfully extended to unstructured meshes~\cite{DingCuiWu2024}, adapted to Hermite WENO schemes~\cite{FanWu2024}, and applied to a variety of hyperbolic systems~\cite{CaoPengWu,LiuWu2024,YanAbgrallWu} and related frameworks~\cite{li2024spectral,peng2025oscillation}.
	
	To address the above numerical challenges, we propose a high-order {CP} compact OEDG (CPcOEDG) method for spherically symmetric EE systems. The main theoretical and algorithmic contributions are summarized as follows:
	\begin{itemize}
		\item \textbf{Characterization of admissible states in conservative variables.}  
		We derive a necessary and sufficient condition for physical admissibility of the fluid directly in terms of the conservative vector $\boldsymbol{U}$. In contrast to standard approaches that verify constraints in primitive variables---a process complicated by relativistic nonlinearities and implicit mappings---we prove that the admissible set $G_c$ is convex and can be explicitly characterized as
		\[
		G_c=\Big\{\boldsymbol{U}=(\mathcal{T}^{00},\mathcal{T}^{01})^\top:~ \mathcal{T}^{00}-\lvert\mathcal{T}^{01}\rvert>0\Big\}.
		\]
		This explicit convexity provides the theoretical foundation for rigorous constraint preservation throughout the numerical evolution.

		\item \textbf{GQL-based flux inequalities.}
		To establish the {CP} property, we employ the Geometric Quasi-Linearization (GQL) framework \cite{Wu2023Geometric,Wu2017RMHD,Wu2018A} to derive the following key flux inequalities:
		\begin{align*}
			\boldsymbol F(\boldsymbol U)\cdot \boldsymbol n 
			&> s_1(\boldsymbol U)\,\boldsymbol U\cdot \boldsymbol n, \qquad
			\boldsymbol F(\boldsymbol U)\cdot \boldsymbol n 
			< s_2(\boldsymbol U)\,\boldsymbol U\cdot \boldsymbol n,
		\end{align*}
		where the flux function $\boldsymbol F(\boldsymbol U)$ is defined in \eqref{eq:UFS}, and $s_1$ and $s_2$ are given in \eqref{equ:eigenF}. 
		This formulation compensates for the breakdown of the standard Lax--Friedrichs (LF) splitting property in the present context. While LF splitting is a widely used tool for proving CP properties requiring $\boldsymbol U \pm \frac{1}{{s_i}}\,{\boldsymbol F}(\boldsymbol U) \in G_c$, we show in \cref{lem:LFviolation} that this condition seems not applicable for the EE system with Harten--Lax--van Leer (HLL) flux. The derived GQL inequalities provide the alternative bounds needed for a rigorous stability analysis.
	
		\item \textbf{Rigorous preservation of physical constraints for barotropic fluids.} 
		Stability analysis of the EE system is often hindered by the implicit nature of general barotropic equations of state. Building upon the convex structure of $G_c$, we establish a theoretical guarantee that the numerical solution remains physically admissible. Specifically, we prove that the DG spatial operator, combined with the Zhang--Shu limiter \cite{zhang2011maximum}, satisfies the \emph{weak constraint-preserving} property under forward-Euler time stepping. Following the convex decomposition technique from \cite{LiuSunZhang2025,ChenSunXing2024}, we extend our theoretical analysis to the general compact RKDG case, thereby guaranteeing positive density and subluminal velocity for arbitrary high orders under suitable CFL conditions.
		
		\item \textbf{Intrinsic enforcement of metric constraints.} 
		To enforce the geometric constraints on the spacetime metric (namely, metric potentials $A \in (0,1)$ and $B \in (0,\infty)$), we introduce a bijective change of variables. Rather than evolving $A$ and $B$ directly, we evolve unconstrained auxiliary variables $Y, Z \in \mathbb{R}$ via the transformation
		\[
		Y=\tfrac{1}{8}\ln\!\Big(\tfrac{1}{A}-1\Big)+\tfrac{1}{2}, \qquad Z=\ln(B).
		\]
		This mapping ensures that $A$ and $B$ automatically satisfy their strict bounds for any real-valued numerical solution of $(Y,Z)$, effectively eliminating the need for artificial metric limiters.
		
		\item \textbf{Metric regularity and global continuity.} 
		We propose a hybrid update strategy to maintain the smoothness and consistency of the spacetime geometry. The variable $Y$ is evolved via a pointwise ODE at element nodes, while $Z$ is recovered by spatial integration of the Einstein constraint equation~\eqref{eq:rhd3}. 
		By employing high-order numerical fluxes at cell interfaces for the spatial integration, the reconstructed metric components---specifically, $A$ at each time step and $B$ across all RK stages---remain globally continuous.
		
		\item \textbf{Order-preserving boundary treatment.} 
		In standard RKDG schemes, exactly enforcing time-dependent Dirichlet boundary conditions at each RK stage typically induces order reduction. We address this by combining the compact RKDG framework \cite{ChenSunXing2024} with the compatible boundary condition technique \cite{CarpenterGottliebAbarbanelDon1995}. This boundary treatment imposes data consistently with all intermediate stage updates, thereby preserving the full design-order accuracy of the method.
		
		\item \textbf{Numerical validation.} 
		Extensive numerical experiments demonstrate that our scheme retains high-order accuracy while robustly preserving physical constraints. The method effectively eliminates nonphysical oscillations and achieves designed convergence rate, providing a reliable tool for the simulation of astrophysical phenomena such as neutron stars and black hole accretion.
	\end{itemize}
	
	The paper is organized as follows: \cref{sec:gov} introduces the governing equations of the EE system in spherically symmetric general relativity and the associated admissible set together with its equivalent conservative formulation; \cref{sec:scheme} details the numerical scheme and presents the CFL condition that guarantees the CP property; \cref{sec:experiments} reports numerical results; and conclusions are given in \cref{sec:conclusions}. Throughout the paper, the speed of light $c$ is normalized to $1$, and repeated upper–lower index pairs are summed over according to the Einstein summation convention. 
	
	\section{Governing equations}\label{sec:gov}
	\subsection{The EE equations}
	The local conservation laws of the relativistic Euler equations can be expressed
	in terms of the four-current \(J^\mu\) and the stress-energy tensor \(T^{\mu\nu}\).
	With Greek indices \(\mu,\nu=0,1,2,3\), an ideal fluid satisfies
	\[
	J^\mu=\rho_0 u^\mu,\qquad
	T^{\mu\nu}=(\rho+p)u^\mu u^\nu + pg^{\mu\nu},
	\]
	where \(\rho_0\) is the rest-mass density, \(e\) is the specific internal energy,
	and \(\rho=\rho_0(1+e)\) is the rest energy density in the comoving frame.
	The quantity \(p\) denotes the pressure, \(u^\mu\) the fluid four-velocity, and \(g_{\mu\nu}\) the spacetime metric with line element
	\(
	\mathrm{d}s^2=g_{\mu\nu}\,\mathrm{d}x^\mu \mathrm{d}x^\nu,
	\)
	satisfying \(g^{\mu\lambda}g_{\lambda\nu}=\delta^\mu_{\ \nu}\). 
	
	The corresponding local conservation laws are 
	\begin{align}\label{eq:DJT}
		\nabla_\mu J^\mu &= 0, \\
		\nabla_\mu T^{\mu\nu} &= 0, \nonumber
	\end{align}
	where $\nabla_\mu$ denotes the covariant derivative associated with $g_{\mu\nu}$. 
	To close the system, we prescribe both the spacetime geometry and an equation of state:
	\begin{itemize}
		\item \textbf{Spacetime geometry.}
		The metric is determined by Einstein’s field equations
		\begin{equation}\label{eq:distr}
			R^{\mu\nu}-\tfrac12 g^{\mu\nu}R=\kappa T^{\mu\nu},
		\end{equation}
		where $R^{\mu\nu}$ is the Ricci tensor, $R$ is the scalar curvature, and $\kappa=8\pi$ is the Einstein coupling constant. We work in units with Newton’s gravitational constant $\mathcal{G}=1$.
		
		\item \textbf{Equation of state.}
		We assume a barotropic relation
		\begin{equation}\label{eq:pRho}
			p=P(\rho),
		\end{equation}
		satisfying
		\begin{equation}\label{eq:ass_p}
			P(0)=0,\qquad 0<P'(\rho)<1,\qquad 	P(\rho) < \sqrt{P'(\rho)}\,\rho.
		\end{equation}
		The additional requirement $P(\rho) < \sqrt{P'(\rho)}\,\rho$  is crucial for our {CP} analysis and is sharp for the GQL-based flux inequalities in \cref{lem:VG2}.
		\begin{remark}
			Condition \eqref{eq:ass_p} automatically implies 
			\begin{equation}\label{eq:key_re}
				p<\rho,\quad \textrm{sign}(p) =\textrm{sign}(\rho), \quad \frac{\rho}{p}>1 \text{ if } p>0.
			\end{equation}
			This relation plays a central role in the analysis of barotropic equations of state.
		\end{remark}
		
	\end{itemize}

	\subsection{Spherically symmetric EE system}
	In standard Schwarzschild coordinates, the spherically symmetric gravitational metric \cite{VoglerTemple2012} takes the form
	\begin{equation}\label{eq:metric}
		\mathrm{d}s^2=-B\left(t,r\right)\mathrm{d}t^2+\frac{1}{A\left(t,r\right)}\mathrm{d}r^2+r^2\left(\mathrm{d}\theta^2+\sin^2{\theta}\mathrm{d}\varphi^2\right),
	\end{equation}
	where $A:=1-\frac{2M}{r} \in (0,1)$, $B \in (0,\infty)$, $\sqrt{B(t, r)}$ is the lapse function, and $M$ is the mass function. Here, the coordinates $(t, r)$ denote temporal and radial components, respectively, forming the spacetime coordinate system $x = (x^0, x^1, x^2, x^3) = (t, r, \theta, \varphi)$. 
	
	If we further assume that this metric is Lipschitz and that the stress-energy tensor is bounded in $L^\infty$ \cite{WuTang2016}, then the EE system \cref{eq:DJT,eq:pRho,eq:distr} is weakly equivalent to
	\begin{align}
		\frac{\partial\boldsymbol{U}}{\partial t}+\frac{\partial(\sqrt{AB}\boldsymbol{F}(\boldsymbol{U}))}{\partial r}&=\boldsymbol{S}(\boldsymbol{U},A,B,r),\label{eq:rhd1}\\
		\frac{\partial M}{\partial t}&=-\frac{1}{2} \kappa r^2 \sqrt{A B} \mathcal{T}^{01},\label{eq:rhd2}\\
		\frac{1}{B}\frac{\partial B}{\partial r}&=\frac{1-A}{Ar}+\frac{\kappa r}{A}\mathcal{T}^{11},\label{eq:rhd3}
	\end{align}
	where the conserved variables $\boldsymbol{U}$, flux $\boldsymbol{F}$, and source term $\boldsymbol{S}$ are defined by
	\begin{align}\label{eq:UFS}
		&\qquad\qquad \boldsymbol{U}=\left(\mathcal{T}^{00},\mathcal{T}^{01}\right)^\top,\quad\boldsymbol{F}=\left(\mathcal{T}^{01},
		\mathcal{T}^{11}\right)^\top,\\
		\boldsymbol S &= -\sqrt{AB}\begin{pmatrix}
			\dfrac{2}{r}\,\mathcal T^{01}\\[6pt]
			\dfrac{2}{r}\,\mathcal T^{11}
			+ \dfrac{1-A}{2Ar}\,\big(\mathcal T^{00}-\mathcal T^{11}\big)
			+ \dfrac{\kappa r}{A}\big(\mathcal T^{00}\mathcal T^{11}-(\mathcal T^{01})^{2}\big)
			- \dfrac{2p}{r}\end{pmatrix}.\nonumber
	\end{align}
	Let  $\boldsymbol{V} = (\rho, p, v)^\top$ be the primitive variables with velocity $v= \frac{1}{\sqrt{AB}}\frac{u^1}{u^0}$ and $W = \frac{1}{\sqrt{1-v^2}}$ the Lorentz factor. The components of the stress-energy tensor are
	\begin{align}\label{eq:s_etensor}
		\mathcal{T}^{00} = (\rho + p)W^2 - p, \quad \mathcal{T}^{01} = (\rho + p)W^2 v, \quad \mathcal{T}^{11} = (\rho v^2 + p)W^2.
	\end{align}

	\subsection{Physical constraints of fluid variables in EE equations}
	The admissible set $G_p$ for the hydrodynamic component of the EE system, written in terms of the primitive variables $\boldsymbol{V}$, is defined as
	\begin{equation} \label{equ:defGp}
		G_p = \left\{ (\rho, p, v)^\top \mid \rho > 0, \, p > 0, \, |v| < 1 \right\}.
	\end{equation}
	Our goal is to show that this admissible set can be equivalently characterized in terms of the conservative variables, denoted by $G_c$ in \eqref{equ:defGc}. As a first step, we establish the following lemma describing the relation between velocity and momentum.

	\begin{lemma}\label{lem:equi_v}
		Let $u := \mathcal{T}^{01} / \mathcal{T}^{00}$. If $|u| < 1$ and condition \eqref{eq:key_re} holds (implying $\mathrm{sign}(\rho) = \mathrm{sign}(p)$), then the physical velocity $v$ is uniquely given by
		\begin{equation} \label{eq:v_solution}
			v(u) = \frac{2u}{1 + \frac{p}{\rho} + \sqrt{\left(1+\frac{p}{\rho}\right)^2 - 4\frac{p}{\rho}u^2}}.
		\end{equation}
	\end{lemma}

	\begin{proof}
		From the definition of the stress-energy tensor in \cref{eq:s_etensor}, we obtain the relation
		\begin{equation*}
			p u v^2 - (\rho + p)v + \rho u = 0, \quad \text{with } u := \frac{\mathcal{T}^{01}}{\mathcal{T}^{00}}.
		\end{equation*}
		Solving this quadratic equation for $v$ yields two roots:
		\begin{align}
			v_{\pm}(u)&=\frac{\rho+p\pm\sqrt{\left(\rho+p\right)^2-4\rho pu^2}}{2pu}.\label{eq:vu_pm}
		\end{align}
		To identify the physically admissible root, we substitute \eqref{eq:s_etensor} into \cref{eq:vu_pm} and rewrite \cref{eq:vu_pm} as  
		\[
		v_{\pm}(u)=\frac{\rho+p\pm\sqrt{\left(\rho+p\right)^2-4\rho p\left(\frac{(\rho+p)v}{\rho+pv^2}\right)^2}}{2p\frac{(\rho+p)v}{\rho+pv^2}}.
		\]
		This can be further simplified using a perfect square identity:
		\begin{align}
			v_{\pm}(u)&=\frac{\rho+p\pm\left|\frac{(\rho-pv^2)(\rho+p)}{\rho+pv^2}\right|}{2p\frac{(\rho+p)v}{\rho+pv^2}}
			=\frac{1}{2v}\left(\frac{\rho}{p}+v^2\pm\left|\frac{\rho}{p}-v^2\right|\right).\nonumber
		\end{align}
		Since we restrict to physically meaningful velocities $v \in (-1,1)$, we have  
		\[
		v_{\pm}(u)=\frac{1}{2v}\left(\frac{\rho}{p}+v^2\pm\left(\frac{\rho}{p}-v^2\right)\right),
		\]
		i.e., $v_{+}(u)={\rho}/(pv)$ and $v_{-}(u)=v$. According to \eqref{eq:key_re}, $v_{+}(u)>1$. Therefore, only $v_{-}(u)$ is admissible and the velocity is uniquely given by
		\begin{align}
			v(u)&=\frac{\rho+p-\sqrt{\left(\rho+p\right)^2-4\rho pu^2}}{2pu}=\frac{2u}{1+\frac{p}{\rho}+\sqrt{\left(1-\frac{p}{\rho}\right)^2+4(1-u^2)\frac{p}{\rho}}}.\label{eq:vu}
		\end{align}
	\end{proof}
	
	We next show that the constraint $|v|<1$ is equivalent to $|u|<1$.
	\begin{lemma}\label{lem:uvequi}
		The function $v(u)$ is monotone for $u\in(-1,1)$ and satisfies $v(1) = 1$ and $v(-1)= -1$. Consequently,
		\[
		|v|<1 \quad\text{is equivalent to}\quad |u|<1.
		\]
	\end{lemma}
	
	\begin{theorem}\label{the:G}
		The admissible set $G_p$ is equivalent to the set of conservative variables $G_c$, defined by
		\begin{equation}
			\label{equ:defGc}	
			G_c=\left\{\boldsymbol{U}=\left(\mathcal{T}^{00}, \mathcal{T}^{01}\right)^\top:~ \boldsymbol{U}\cdot\boldsymbol{n}>0 \text{ with } \boldsymbol{n}=(1,\pm1)^\top\right\}.
		\end{equation}
		Furthermore, the set $G_c$ is convex.
	\end{theorem}

	\begin{proof} 
		The convexity of $G_c$ follows immediately from the linearity of the inequalities defining it and is therefore omitted. We focus on proving the equivalence $G_p = G_c$.
		
		\medskip\noindent
		\textbf{Inclusion $G_p \subset G_c$.}  
		Consider a primitive state $\boldsymbol{V} = (\rho, p, v)^\top \in G_p$. Using the ideal fluid relations, we compute
		\begin{equation}\label{eq:T00T01}
			\mathcal{T}^{00}(\boldsymbol{V})-\left|\mathcal{T}^{01}(\boldsymbol{V})\right|
			=\frac{\rho+p|v|^2-(\rho+p)|v|}{1-v^2}
			=\frac{p\left(\frac{\rho}{p}-|v|\right)}{1+|v|}.
		\end{equation}
		Since $p < \rho$ and $|v| < 1$, the right-hand side is strictly positive. Hence, $\boldsymbol{U}(\boldsymbol{V}) \in G_c$. 
		
		\medskip\noindent
		\textbf{Inclusion $G_c \subset G_p$.}  
		Conversely, assume $\boldsymbol{U} \in G_c$. Then $|u|<1$ holds and, by Lemma~\ref{lem:uvequi}, this implies $|v|<1$. Furthermore, we have
		\begin{equation}
			\mathcal{T}^{00} - |\mathcal{T}^{01}| = \frac{\rho - p|v|}{1+|v|} > \frac{p(1-|v|)}{1+|v|}
			> 0,
		\end{equation}
		which guarantees $p>0$. The positivity of $\rho$ then follows from the fact that $p$ and $\rho$ share the same sign. Thus $\boldsymbol{V} \in G_p$, completing the proof.
	\end{proof}

	\subsection{Metric constraints and auxiliary variables}
	To ensure the Lorentzian signature $(-, +, +, +)$ of the metric \eqref{eq:metric}, the metric potentials must strictly satisfy
	\begin{equation} \label{eq:metric_constraints}
		A(t,r) \in (0, 1) \quad \text{and} \quad B(t,r) \in (0, \infty).
	\end{equation}

	\subsubsection*{Automatic constraint preservation} 
	We propose an approach that evolves unconstrained auxiliary variables $Y$ and $Z$ instead of the constrained variables $M$ and $B$. This bijective mapping ensures that the metric constraints are intrinsically satisfied via the inverse transformations, thereby removing the need for any post-processing limiters or cut-off techniques.
	
	\begin{itemize}
		\item \textbf{Transformation between $Y$ and $M$.}  
		We define the auxiliary variable $Y$ through
		\begin{equation}
			Y = -\frac{1}{8} \ln\left(\frac{r}{2M} - 1\right) + \frac{1}{2}.
		\end{equation}
		The inverse transformation is explicitly given by
		\begin{equation}
			M = \frac{r}{2\left(1 + e^{-8(Y - 1/2)}\right)} \in \left(0, \frac{r}{2}\right).
		\end{equation}
		Thus, for any real-valued numerical solution $Y$, the mass function $M$ remains strictly within $(0, r/2)$, which in turn ensures $A = 1 - 2M/r \in (0, 1)$.
		
		\item \textbf{Transformation between $Z$ and $B$.}  
		We set $Z = \ln(B)$, so that the inverse mapping $B = e^Z$ guarantees $B \in (0, \infty)$ for all real $Z$.
	\end{itemize}
	
	Since $A$ and $B$ are assumed Lipschitz continuous, the variables $M$, $Y$, and $Z$ are also continuous. Consequently, the evolution equations \eqref{eq:rhd2}--\eqref{eq:rhd3} for $A$ and $B$ can be reformulated in terms of $Y$ and $Z$, and the resulting spherically symmetric EE system is given by \eqref{eq:rhd-YZ-alt} in a weak sense.
	
	\section{Numerical scheme}\label{sec:scheme}
	In this section, we present the numerical framework for solving the EE system on the domain {$\Omega = [r_L, r_R]$}:
	\begin{equation}\label{eq:rhd-YZ-alt}
		\begin{aligned}
			\frac{\partial\boldsymbol{U}}{\partial t}+\frac{\partial(\sqrt{A(Y)B(Z)}\boldsymbol{F}(\boldsymbol{U}))}{\partial r}&=\boldsymbol{S}(\boldsymbol{U},A(Y),B(Z),r),\\
			\partial_t Y &= -\frac{\kappa r}{8}\,\sqrt{\frac{A(Y)}{B(Z)}}\,\frac{\mathcal T^{01}}{1-A},\\[0.6ex]
			\partial_r Z &= \frac{1-A(Y)}{A(Y)\,r} + \frac{\kappa r}{A(Y)}\,\mathcal T^{11}.
		\end{aligned}
	\end{equation}
	The framework consists of the cOEDG scheme for $\boldsymbol{U}$, the discretization and continuous reconstruction of $Y$ and $Z$, and the enforcement of fluid constraints.
	Note that in \eqref{eq:rhd-YZ-alt}, only the evolution of the conservative variables $\boldsymbol{U}$ involves convection; accordingly, we discretize $\boldsymbol{U}$ and $(Y,Z)^\top$ using different strategies. 
	
	\subsection{The cOEDG method for \(\boldsymbol{U}\)}
	The cOEDG scheme combines a componentwise OEDG damping mechanism \cite{LiuWu2024} with the compact Runge--Kutta DG (cRKDG) formulation. The latter blends the standard DG spatial operator with a local operator to maintain compactness and high order under nonhomogeneous Dirichlet conditions~\cite{Gustafsson1975}. To preserve compactness, the cOEDG method applies the OE step only at the final stage of the cRKDG update.
	
	We first recall the conventional DG and local spatial discretizations for $\boldsymbol{U}$. Let $I_j =(r_{j-\frac{1}{2}}, r_{j+\frac{1}{2}})$, $j = 1, \ldots, N$, form a partition of the domain $\Omega$. We define the finite element space
	\begin{displaymath}
		\mathbb{V}_h^k = \left\{ q \in L^2(\Omega) : \left. q \right|_{I_j} \in \mathbb{P}^k(I_j), \quad j = 1, \ldots, N \right\},
	\end{displaymath}
	where $\mathbb{P}^k(I_j)$ denotes the space of polynomials of degree at most $k$ on the cell $I_j$. The conventional and local DG methods seek a numerical solution $\boldsymbol{U}_h\in [\mathbb{V}_h^k]^2$ satisfying
	\begin{align}\label{eq:hD_DG}
		&\int_{I_j} ({\boldsymbol{U}_h})_t\cdot \boldsymbol{\phi } \, {\mathrm{d}r} =\mathcal{L}_j^{\text{DG}/\text{loc}}({\boldsymbol{U}_h},A_h,B_h; \boldsymbol{\phi }),\quad \forall \boldsymbol{\phi } \in [\mathbb{V}_h^k]^2, 
	\end{align}
	with 
	\begin{align*}
		\mathcal{L}_j^{\text{DG}}({\boldsymbol{U}_h}, A_h,B_h;\boldsymbol{\phi }) &:= \int_{I_j}\sqrt{A_hB_h} \boldsymbol{F}(\boldsymbol{U}_h) \cdot\partial_r \boldsymbol{\phi } \, \mathrm{d}r - \left[ \widehat{\left(\sqrt{A B} {\boldsymbol{F}}\right)}_{j+\frac{1}{2}}\cdot \boldsymbol{\phi }\left(r_{j+\frac{1}{2}}^-\right)\right.\\
		&\left. - \widehat{\left(\sqrt{A B} {\boldsymbol{F}}\right)}_{j-\frac{1}{2}} \cdot\boldsymbol{\phi }\left(r_{j-\frac{1}{2}}^+\right) \right] + \int_{I_j} \boldsymbol{S}(\boldsymbol{U}_h, A_h, B_h, r)  \cdot\boldsymbol{\phi }\, \mathrm{d}r, \nonumber
	\end{align*}
	and 
	\begin{equation*}
		\resizebox{0.92\hsize}{!}{$
			\begin{aligned}
				\mathcal{L}_j^{\text{loc}}({\boldsymbol{U}_h},A_h,B_h; \boldsymbol{\phi}) &:= \int_{I_j} \sqrt{A_hB_h}\boldsymbol{F}(\boldsymbol{U}_h) \cdot \partial_r \boldsymbol{\phi } \, {\mathrm{d}r} - \left[ \sqrt{A_{h,j+\frac{1}{2}} B_{h,j+\frac{1}{2}}} {{\boldsymbol{F}}_{j+\frac{1}{2}}} \cdot\boldsymbol{\phi }\left(r_{j+\frac{1}{2}}^-\right)\right.\\
				&\left. - \sqrt{A_{h,j-\frac{1}{2}} B_{h,j-\frac{1}{2}}} {{\boldsymbol{F}}_{j-\frac{1}{2}}} \cdot\boldsymbol{\phi }\left(r_{j-\frac{1}{2}}^+\right) \right] + \int_{I_j} \boldsymbol{S}(\boldsymbol{U}_h, A_h, B_h, r)  \cdot\boldsymbol{\phi } \, {\mathrm{d}r}, \nonumber
			\end{aligned}$}
	\end{equation*}
	where \(\boldsymbol{F}\) is the physical flux and 
	\(\widehat{\left(\sqrt{A B} {\boldsymbol{F}}\right)}_{j+\frac{1}{2}}
	= \widehat{\left(\sqrt{A B} {\boldsymbol{F}}\right)}\!\left(\boldsymbol{U}_{j+\tfrac12}^{\pm},A_{j+\tfrac12}^{\pm},B_{j+\tfrac12}\right)\)
	denotes the HLL numerical flux at the interface \(r_{j+\tfrac12}\), where
	\begin{equation*}
		\resizebox{1\hsize}{!}{$
			\begin{aligned}
				\widehat{\left(\sqrt{A B} {\boldsymbol{F}}\right)}\!\left(\boldsymbol{U}_{j+\tfrac12}^\pm,A_{j+\frac12}^\pm,B_{j+\tfrac12}\right)=&\frac{\alpha^R_{j+\frac{1}{2}}(\sqrt{A B} {\boldsymbol{F}})_{j+\frac12}^--\alpha^L_{j+\frac{1}{2}}(\sqrt{A B} {\boldsymbol{F}})_{j+\frac12}^+}{\alpha^R_{j+\frac12}-\alpha^L_{j+\frac12}}\\
				&+\frac{\alpha^L_{j+\frac{1}{2}}\alpha^R_{j+\frac{1}{2}}\left(\boldsymbol{U}_{j+\frac{1}{2}}^+-\boldsymbol{U}_{j+\frac{1}{2}}^-\right)}{\alpha^R_{j+\frac12}-\alpha^L_{j+\frac12}},\nonumber
			\end{aligned}$}
	\end{equation*}
	with
	\begin{subequations}
		\label{equ:sLsR}
		\begin{align}
			\alpha^L_{j+\frac{1}{2}}&=\min{\left(\lambda_1\left(\boldsymbol{U}_{j+\frac{1}{2}}^-\right),\lambda_1\left(\boldsymbol{U}_{j+\frac{1}{2}}^+\right),0\right)},\\
			\alpha^R_{j+\frac{1}{2}}&=\max{\left(\lambda_2\left(\boldsymbol{U}_{j+\frac{1}{2}}^-\right),\lambda_2\left(\boldsymbol{U}_{j+\frac{1}{2}}^+\right),0\right)}.
		\end{align}
	\end{subequations}
	Here, $\lambda_1$ and $\lambda_2$ denote the smallest and largest eigenvalues of the Jacobian matrix $\partial \left(\sqrt{A B} {\boldsymbol{F}}\right)/\partial \boldsymbol{U}$ (see \cref{app:nf} for explicit formulas). 
	All cell integrals are evaluated by Gauss quadrature; that is, for any \(f\),
	\[
	\int_{I_j} f(r)\,{\mathrm{d}r} \;\approx\; \sum_{l=0}^{N^{\mathrm{GQ}}} w_{j,l}\, f\!\left(r_{j,l}^{\mathrm{GQ}}\right),
	\]
	where \(\{w_{j,l}\}\) and \(\{r_{j,l}^{\mathrm{GQ}}\}\) are the Gauss weights and nodes on \(I_j\). The values of \(A\) and \(B\) at the Gauss points are obtained by
evaluating the interpolating polynomials constructed
from their nodal values in \cref{subsec:YZupdate} , whereas the values of \(\boldsymbol{U}\) at the
Gauss points are computed directly from the numerical solution.

	Consider an explicit RK method of order $r$ with $s$ stages, associated with the Butcher tableau
	\[
	\begin{array}{c|c}
		\mathbf{c} & L \\
		\hline
		& \mathbf{b}
	\end{array}
	\quad \text{with} \quad L = (a_{im})_{s \times s}, \quad \mathbf{b} = (b_1, \ldots, b_s),
	\]
	where \(L\) is a strictly lower triangular matrix. The corresponding cRKDG scheme is
	\begin{subequations}
		\begin{align}\label{eq:rk}
			{\boldsymbol{U}}_\sigma^{n}  & = \mathcal{F}_\tau {\boldsymbol{U}}_h^{n}, \\
			\int_{I_j} \boldsymbol{U}_\sigma^{n,i}  \cdot\boldsymbol{\phi } \, {\mathrm{d}r}&= \int_{I_j} \boldsymbol{U}_\sigma^{n}  \cdot\boldsymbol{\phi } \, {\mathrm{d}r} + \Delta t\sum_{m=0}^{i-1} a_{im} \mathcal{L}_j^{\mathrm{loc}} \left( \boldsymbol{U}_\sigma^{n,m},A_h^{n,m},B_h^{n,m}; \boldsymbol{\phi } \right), \\
			\int_{I_j} \boldsymbol{U}_h^{n,s}  \cdot\boldsymbol{\phi }\, {\mathrm{d}r}&= \int_{I_j} \boldsymbol{U}_\sigma^{n}  \cdot\boldsymbol{\phi }  {\,\mathrm{d}r}+\Delta t  \sum_{i=0}^{s-1} b_i\mathcal{L}_j^{\mathrm{DG}} \left( \boldsymbol{U}_\sigma^{n,i},A_h^{n,i},B_h^{n,i}; \boldsymbol{\phi } \right),\\
			\boldsymbol{U}_h^{n+1} &= \boldsymbol{U}_h^{n,s}.
		\end{align}
	\end{subequations}
	Here, the updates of $A_h(Y)$ and $B_h(Z)$ are postponed until \cref{subsec:YZupdate}.
	\begin{remark}
		By replacing the local spatial operator $\mathcal{L}_j^{\text{loc}}$ with the DG spatial operator $\mathcal{L}_j^{\text{DG}}$ and inserting the OE procedure after each RK stage, we recover the conventional OEDG method
		\begin{subequations}
			\begin{align*}\label{eq:rkoedg}
				{\boldsymbol{U}}_\sigma^{n}  & = \mathcal{F}_\tau {\boldsymbol{U}}_h^{n}, \\
				\int_{I_j} \boldsymbol{U}_h^{n,i}  \cdot\boldsymbol{\phi } \, {\mathrm{d}r}&= \int_{I_j} \boldsymbol{U}_\sigma^{n}  \cdot\boldsymbol{\phi } \, {\mathrm{d}r} + \Delta t\sum_{m=0}^{i-1} a_{im} \mathcal{L}_j^{\mathrm{DG}} \left( \boldsymbol{U}_\sigma^{n,m},A_h^{n,m},B_h^{n,m}; \boldsymbol{\phi } \right), \\[-0.3em]
				{\boldsymbol{U}}_\sigma^{n,i}  & = \mathcal{F}_\tau {\boldsymbol{U}}_h^{n,i},\\
				\int_{I_j} \boldsymbol{U}_h^{n,s}  \cdot\boldsymbol{\phi }\, {\mathrm{d}r}&= \int_{I_j} \boldsymbol{U}_\sigma^{n}  \cdot\boldsymbol{\phi }  {\,\mathrm{d}r}+\Delta t  \sum_{i=0}^{s-1} b_i\mathcal{L}_j^{\mathrm{DG}} \left( \boldsymbol{U}_\sigma^{n,i},A_h^{n,i},B_h^{n,i}; \boldsymbol{\phi } \right),\\
				\boldsymbol{U}_h^{n+1} &= \boldsymbol{U}_h^{n,s}.
			\end{align*}
		\end{subequations}
	\end{remark}
	The operator \(\mathcal{F}_\tau\) maps \(\boldsymbol U_h\) to \(\mathcal{F}_\tau\boldsymbol U_h=\boldsymbol U_\sigma(\cdot,\tau)\) through a pseudo-time \(\hat t\) damping ODE for \(\boldsymbol U_\sigma(r,\hat t)\) with initial data \(\boldsymbol U_h^n\):
	\begin{equation}\label{eq:OE-ode}
		\resizebox{0.99\hsize}{!}{$
			\begin{cases}
				\dfrac{d}{d\hat t}\displaystyle \int_{I_j} \boldsymbol U_\sigma \cdot \boldsymbol\phi \, {\mathrm{d}r}
				\;+\; \displaystyle\sum_{m=0}^{k} \beta_j 
				\int_{I_j} \left[\,\dfrac{\Lambda_j^{m}(\boldsymbol U_h)}{h_j}\bigl(\boldsymbol U_\sigma - P^{m-1}\boldsymbol U_\sigma\bigr)\right]\cdot \boldsymbol\phi \, {\mathrm{d}r} \;=\; 0,
				\quad \forall\, \boldsymbol\phi \in [\mathbb P^{k}(I_j)]^{2},\\[1.0ex]
				\boldsymbol U_\sigma(r,0)=\boldsymbol U_h^{n}(r).
			\end{cases}
			$}
	\end{equation}
	Here, \(\Lambda_j^{m}=\mathrm{diag}\bigl(\sigma_j^{m}(\mathcal T^{00}),\,\sigma_j^{m}(\mathcal T^{01})\bigr)\).
	The approximate local maximal wave speed \(\beta_j\) is given by the spectral radius of
	\(\frac{\partial \sqrt{A B} {\boldsymbol{F}} }{\partial \boldsymbol U}\bigl(\overline{\boldsymbol U}^{\,n}_{h,j},\,\overline{A}_{h,j}^n,\overline{B}_{h,j}^n\bigr)\), where the cell averages are $\overline{\boldsymbol U}^{\,n}_{h,j}:=\frac{1}{|I_j|}\int_{I_j} \boldsymbol{U}^{\, n}_{h,j} \mathrm{d}x$, $\,\overline{A}_{h,j}^n$, and $\overline{B}_{h,j}^n$.
	The operator \(P^{m}\) is the standard \(L^{2}\) projection onto \([\mathbb P^{m}(I_j)]^{2}\), and \(P^{-1}\) is understood as \(P^{0}\).
	
	Unlike the standard OEDG choice, for $i\in \{0,1\}$ we take componentwise damping coefficients
	\[
	\sigma_j^{m}\!\left(\mathcal T^{0i}_h\right)=
	\begin{cases}
		0, & \text{if } \mathcal T^{0i}_h \equiv \mathrm{avg}(\mathcal T^{0i}_h)\ \text{on } I_j,\\[0.6ex]
		\displaystyle \frac{(2m+1)\, h_j^{m}}{(2k-1)\,m!}\,
		\frac{\sqrt{\bigl[\partial_r^{m}\mathcal T^{0i}_h\bigr]_{j-\frac12}^{2}
				+\bigl[\partial_r^{m}\mathcal T^{0i}_h\bigr]_{j+\frac12}^{2}}}
		{\ \bigl\|\mathcal T^{0i}_h-\mathrm{avg}(\mathcal T^{0i}_h)\bigr\|_{L^\infty(\Omega)}}, & \text{otherwise},
	\end{cases}
	\]
	where \([w]_{j+\frac12}=w(r_{j+\frac12}^{+})-w(r_{j+\frac12}^{-})\) denotes the jump at the interface \(r=r_{j+\frac12}\).
	
	\begin{remark} \label{rem:OEDGandcOEDG}
		The original OEDG methods employ damping coefficients based on the maximum over all components. However, in scenarios such as \cref{ex:TOV}, the initial momentum satisfies $\mathcal{T}_h^{01} \equiv 0$, which contradicts the nonconstant assumption used in the optimal error analysis of the original OEDG theory \cite{PengSunWu2025}. In such cases, the standard approach yields $\max_j\sigma_j^m \approx \mathcal{O}(h^{1/2})$, which can degrade the observed accuracy. 
		
		To avoid this, we adopt a \emph{componentwise} damping strategy. This ensures that whenever the damping coefficient for a given component scales at a lower order, the associated residual term $\int (\mathcal{T}_h^{0i} - P^{m-1}\mathcal{T}_h^{0i})\phi \, {\mathrm{d}r}$ is of sufficiently high order to compensate. As a result, the damping term remains high-order and does not destroy the design-order accuracy of the scheme.
	\end{remark}

	\begin{remark}
		Consider an orthonormal basis \(\{\Phi_{j}^{(m)}\}_{m=0}^{k}\) of \(\mathbb{P}^{k}(I_{j})\) and expand
		\[
		\mathcal{T}_{h}^{0i}=\sum_{m=0}^{k}\bigl(\mathcal{T}_{h}^{0i}\bigr)^{(m)}\,\Phi_{j}^{(m)}.
		\]
		The solution of the damping ODE \eqref{eq:OE-ode} admits the explicit form
		\begin{align}
			\mathcal{T}_{\sigma}^{0i}
			= \bigl(\mathcal{T}_{h}^{0i}\bigr)^{(0)} \Phi_{j}^{(0)}
			+ \sum_{m=1}^{k}
			\exp\!\Bigl(-\frac{\beta_{j}\,\tau}{h_{j}}\sum_{l=0}^{m}\sigma_{j}^{\,l}\!\bigl(\mathcal{T}_{h}^{0i}\bigr)\Bigr)\,
			\bigl(\mathcal{T}_{h}^{0i}\bigr)^{(m)} \Phi_{j}^{(m)}. \nonumber
		\end{align}
		Equivalently, in terms of modal coefficients,
		\[
		\bigl(\mathcal{T}_{\sigma}^{0i}\bigr)^{(0)}=\bigl(\mathcal{T}_{h}^{0i}\bigr)^{(0)},\qquad
		\bigl(\mathcal{T}_{\sigma}^{0i}\bigr)^{(m)}=\exp\!\Bigl(-\frac{\beta_{j}\,\tau}{h_{j}}\sum_{l=0}^{m}\sigma_{j}^{\,l}\!\bigl(\mathcal{T}_{h}^{0i}\bigr)\Bigr)\,
		\bigl(\mathcal{T}_{h}^{0i}\bigr)^{(m)},\ \ m\ge1.
		\]
		Thus the OE step acts as a modal filter that preserves the cell average while exponentially damping the higher modes \cite{PengSunWu2025}.
		
	\end{remark}
	
	\subsection{Updates for \((Y,Z)\)}\label{subsec:YZupdate}
	We evolve the pointwise values of the metric variables \((Y,Z)\) at the \(k{+}1\)
	Gauss--Lobatto (GL) nodes \(\{r_{j,\ell}\}_{\ell=0}^{k}\), with
	\(r_{j,0}=r_{j-\frac12}\) and \(r_{j,k}=r_{j+\frac12}\).
	
	\paragraph{Stagewise update for \(Y\)}
	For interior GL nodes \(\ell=1,\dots,k-1\), we apply the RK method to the
	pointwise ODE
	\begin{equation}\label{eq:Y-ode-internal}
		\partial_t Y(r_{j,\ell})
		= -\,\frac{\kappa\, r_{j,\ell}}{8}\,
		\sqrt{\frac{A\!\bigl(Y(r_{j,\ell})\bigr)}{B\!\bigl(Z(r_{j,\ell})\bigr)}}\,
		\frac{\mathcal T^{01}(r_{j,\ell})}{1- A\!\bigl(Y(r_{j,\ell})\bigr)}.
	\end{equation}
	At the cell interfaces \(\ell\in\{0,k\}\), all intermediate RK
	stages use the same pointwise formula \eqref{eq:Y-ode-internal}. Only in the \emph{final} RK
	stage do we evaluate the right-hand side using interface numerical states from an HLL scheme. For instance, at the right interface
	\(r_{j,k}=r_{j+\frac12}^-\), we set
	\begin{equation}\label{eq:Y-ode-interface}
		\partial_t Y_{j+\frac12}^{-}
		= -\,\frac{\kappa\, r_{j+\frac12}}{8}\,
		\sqrt{\frac{\widehat A_{j+\frac12}}{B\!\bigl(Z_{j+\frac12}\bigr)}}\,
		\frac{\widehat{\mathcal T}^{01}_{j+\frac12}}{1-\widehat A_{j+\frac12}}.
	\end{equation}
	Here, \(\widehat{\mathcal T}^{01}\) and \(\widehat A\) are obtained by an HLL
	construction with signal speeds \(\alpha^L\leq0\leq\alpha^R\) (see \eqref{equ:sLsR}):
	\begin{equation}\label{eq:HLL-Y}
		\resizebox{0.98\hsize}{!}{$
			\widehat{\mathcal T}^{01}
			= \frac{\,\alpha^R (\mathcal T^{01})_R - \alpha^L (\mathcal T^{01})_L
				+ \bigl(\sqrt{A B}\,\mathcal T^{11}\bigr)_L
				- \bigl(\sqrt{A B}\,\mathcal T^{11}\bigr)_R\,}{\alpha^R-\alpha^L},
			\quad
			\widehat A
			= \frac{\,\alpha^R A_R - \alpha^L A_L\,}{\alpha^R-\alpha^L}.$}
	\end{equation}
	with \((\cdot)_L\) and \((\cdot)_R\) denoting the left/right traces at the interface,
	and the metric coefficient  \(B(Z)\) being reconstructed from \(Z\).
	
	\paragraph{Interface continuity of \(Y\) in the final stage}
	Assume \((Y^n_{j+\frac12})^+=(Y^n_{j+\frac12})^-\) holds at time \(t_n\) for all
	interfaces. In the final RK stage, we use the same interface states
	\((A_L,A_R)\), \((\mathcal T^{01}_L,\mathcal T^{01}_R)\),
	and \((\sqrt{AB}\,\mathcal T^{11})_{L/R}\) to compute
	\(\widehat A_{j+\frac12}\) and \(\widehat{\mathcal T}^{01}_{j+\frac12}\) via
	\eqref{eq:HLL-Y} on both sides. Hence the right- and left-trace updates apply
	identical increments in \eqref{eq:Y-ode-interface}, which yields
	\[
	\bigl(Y^{n+1}_{j+\frac12}\bigr)^{-}
	= Y^n_{j+\frac12}
	- \Delta t \sum_{i=0}^{s-1} b_i\,
	\frac{\kappa\, r_{j+\frac12}}{8}\,
	\sqrt{\frac{\widehat A^{\,n,i}_{j+\frac12}}{B^{\,n,i}_{j+\frac12}}}\,
	\frac{\widehat{\mathcal T}^{01\,n,i}_{j+\frac12}}
	{1-\widehat A^{\,n,i}_{j+\frac12}}
	= \bigl(Y^{n+1}_{j+\frac12}\bigr)^{+},
	\]
	so \(Y\) remains continuous across the interface at \(t_{n+1}\).

	\paragraph{Continuous updates for \(Z\)}
	Given a boundary or matching condition $Z_h(t,r_m)$, the solution is determined via the line integral
	\begin{equation}{
		\begin{aligned}
			Z_h(t,r_{j,l}) &= Z_h(t,r_{j,l-1})
			+ \int_{r_{j,l-1}}^{r_{j,l}}
			\left(\frac{1 - A_h(t,x)}{A_h(t,x)\,x}
			+ \frac{\kappa x}{A_h(t,x)}\,\mathcal{T}_h^{11}(t,x)\right)dx
			\\ &\approx Z_h(t,r_{j,l-1})+\sum_{m=0}^{N_Z}\gamma_{j,l}^{(m)}\left(\frac{1 - A_h(t,x_{j,l}^m)}{A_h(t,x_{j,l}^{(m)})\,x_{j,l}^{(m)}}
			+ \frac{\kappa x_{j,l}^{(m)}}{A_h(t,x_{j,l}^{(m)})}\,\mathcal{T}_h^{11}(t,x_{j,l}^{(m)})\right),
		\end{aligned}
	}\end{equation}
	where \(\{(\gamma_{j,l}^{(m)},x_{j,l}^{(m)})\}_{m=1}^{N_Z}\) are Gauss weights and nodes on {\([r_{j,l-1},\,r_{j,l}]\)} with $N_Z = \lceil (k+1)/2\rceil$. This integral formulation uniquely determines $Z_h$. Since the interface value $Z_h(t,r_{j-\frac12})$ is shared by neighboring cells, it automatically enforces continuity across cell interfaces.
	
	\subsection{Constraint preservation for the hydrodynamic variables $\boldsymbol{U}$}
	Constraint preservation is essential for the robustness and physical reliability of the computation, and it is also required for the bijectivity of the conservative-to-primitive mapping. Denote the {CP} operator by $\mathcal{B}$. The combined algorithmic framework for the EE system is
	\begin{subequations}
		\begin{align}\label{eq:bprk}
			\boldsymbol{U}_\sigma^{n} &= \mathcal{B}\,\mathcal{F}_\tau \boldsymbol{U}_h^{n},\\
			\int_{I_j} \boldsymbol{U}_h^{n,i}\!\cdot\!\boldsymbol{\phi}\,{\mathrm{d}r}
			&= \int_{I_j} \boldsymbol{U}_\sigma^{n}\!\cdot\!\boldsymbol{\phi}\,{\mathrm{d}r}
			+ \Delta t \sum_{m=0}^{i-1} a_{im}\,
			\mathcal{L}_j^{\mathrm{loc}}\!\left(\boldsymbol{U}_\sigma^{n,m},A_h^{n,m},B_h^{n,m};\boldsymbol{\phi}\right),\\
			\boldsymbol{U}_\sigma^{n,i} &= \mathcal{B}\,\boldsymbol{U}_h^{n,i},\\
			\int_{I_j} \boldsymbol{U}_h^{n,s}\!\cdot\!\boldsymbol{\phi}\,{\mathrm{d}r}
			&= \int_{I_j} \boldsymbol{U}_\sigma^{n}\!\cdot\!\boldsymbol{\phi}\,{\mathrm{d}r}
			+ \Delta t \sum_{i=0}^{s-1} b_i\,
			\mathcal{L}_j^{\mathrm{DG}}\!\left(\boldsymbol{U}_\sigma^{n,i},A_h^{n,i},B_h^{n,i};\boldsymbol{\phi}\right),\\
			\boldsymbol{U}_h^{n+1} &= \boldsymbol{U}_h^{n,s}.
		\end{align}
	\end{subequations}
	Since the OE procedure does not alter the cell average, the CP property of the cOEDG method is identical to that of the underlying cRKDG method. Our CP strategy is built on the GQL-based flux inequalities in \cref{lem:VG2}, which imply the weak CP property (Lemma~3.1 in~\cite{LiuSunZhang2025}) for the three spatial operators in \eqref{eq:threeoperator}. This differs from the usual SSP-RKDG approach, which writes each stage solution as a convex combination of forward-Euler steps with the DG spatial operator. 
	
	\paragraph{Framework of CP for the EE system following \cite{LiuSunZhang2025}}
	\begin{steps}
		\item We first prove the key GQL-based flux inequalities for the EE system.	
		\item Using these inequalities, we show that the weak CP property holds for the three discretizations
		\begin{equation}\label{eq:threeoperator}\mathcal{L}_j^{\text{DG}}, \quad \pm \mathcal{L}_j^{\text{loc}},\end{equation}
		under the classic local scaling limiter, where the scaling coefficient is computed using only local extremal points.
		\item We then prove that the cell averages at the next stage satisfy the constraint via a convex combination of forward-Euler steps with the three spatial discretizations in \eqref{eq:threeoperator}. This yields a CFL condition under which the cell averages remain admissible.
	\end{steps} 
	This CP methodology for compact RK methods \cite{LiuSunZhang2025} bypasses the SSP order barrier and leads to a four-stage, fourth-order bound-preserving cRKDG method. 
	
	\subsubsection*{GQL-based flux inequalities for the EE system}
	\begin{lemma}[GQL-based flux inequalities]\label{lem:VG2}
		If $\boldsymbol U\in G_c$, then
		\begin{align*}
			\boldsymbol F(\boldsymbol U) \cdot \boldsymbol n > s_1(\boldsymbol U)\,\boldsymbol U \cdot \boldsymbol n,\quad 
			\boldsymbol F(\boldsymbol U) \cdot \boldsymbol n < s_2(\boldsymbol U)\,\boldsymbol U \cdot \boldsymbol n,
		\end{align*}
		where
		\begin{equation}\label{equ:eigenF}
			s_1(\boldsymbol U):=\frac{\,v-\sqrt{P'(\rho)}\,}{1-\sqrt{P'(\rho)}\,v}<0,\quad s_2(\boldsymbol U):=\frac{\,v+\sqrt{P'(\rho)}\,}{1+\sqrt{P'(\rho)}\,v}>0
		\end{equation}
		are the eigenvalues of the Jacobian $\partial {\boldsymbol{F}}/\partial \boldsymbol U$, and $\boldsymbol V(\boldsymbol U)=(\rho,p,v)$.
	\end{lemma}
	
	\begin{proof} 
		Let $c_s := \sqrt{P'(\rho)} \in (0, 1)$. Since $|v| < 1$ for any $\boldsymbol{U} \in G_c$, a direct algebraic computation yields the decompositions
		\begin{align}
			\boldsymbol F(\boldsymbol U) - s_1(\boldsymbol U)\boldsymbol U 
			&= \frac{1}{(1-v^2)(1-c_s v)}
			\begin{pmatrix}
				L_1 \\[2pt] L_2
			\end{pmatrix},\nonumber\\
			s_2(\boldsymbol U)\boldsymbol U - \boldsymbol F(\boldsymbol U)
			&= \frac{1}{(1-v^2)(1-c_s v)}
			\begin{pmatrix}
				R_1 \\[2pt] R_2
			\end{pmatrix},\nonumber
		\end{align}
		where
		\begin{align}
			\begin{pmatrix} L_1 \\ L_2 \end{pmatrix}
			:=& (c_s-v)\begin{pmatrix} \rho + p v^2 \\ (\rho+p)v \end{pmatrix}
			\;+\; (1-c_s v)\begin{pmatrix} (\rho+p)v \\ \rho v^2 + p \end{pmatrix},\nonumber\\
			\begin{pmatrix} R_1 \\ R_2 \end{pmatrix}
			:=& (c_s+v)\begin{pmatrix} \rho + p v^2 \\ (\rho+p)v \end{pmatrix}
			\;-\; (1+c_s v)\begin{pmatrix} (\rho+p)v \\ \rho v^2 + p \end{pmatrix}.\nonumber
		\end{align}
		To prove the inequalities, it suffices to verify $L\cdot \boldsymbol{n}>0$ and $R\cdot \boldsymbol{n}>0$. For $L\cdot \boldsymbol{n}$ one checks
		\begin{align*} 
			L_1 \pm  L_2 =& \left(1 \mp v\right)\left(1 \pm v\right)^2\left(c_s\rho \pm p\right).
		\end{align*}
		For $R\cdot \boldsymbol{n}$ we similarly obtain
		\begin{align*} 
			R_1 \pm  R_2 =& (1 \pm v)(1 \mp v)^2\left( c_s \rho\pm p\right). \nonumber
		\end{align*}
		The assumptions \eqref{eq:ass_p} together with the basic inequality \eqref{eq:key_re} guarantee that all factors on the right-hand side are strictly positive.
	\end{proof}
	
	The GQL-based flux inequalities yield the following consequence for the HLL flux.
	
	\begin{corollary}\label{coro:LFsplit}
		Let $(\boldsymbol{U}_i,A_i,B_i)$ with $\boldsymbol U_i\in G_c$ for $i= 1,2$ and let $\widehat{\left(\sqrt{AB}\boldsymbol F\right)}$ be the HLL flux. If $\lambda \max\{-\alpha^L,\alpha^R\}<1$, then
		\begin{align*}\label{equ:coro35}
			\left(\boldsymbol U_2+\lambda\,\widehat{\left(\sqrt{A B} {\boldsymbol{F}}\right)}\!\left(\boldsymbol{U}_{1/2},A_{1/2},B_{1/2}\right)\right)\cdot \boldsymbol{n}&>0,\\
			\left(\boldsymbol U_1-\lambda\,\widehat{\left(\sqrt{A B} {\boldsymbol{F}}\right)}\!\left(\boldsymbol{U}_{2/1},A_{2/1},B_{2/1}\right)\right)\cdot \boldsymbol{n}&>0.
		\end{align*}
	\end{corollary}
	
	\begin{proof}
		By Lemma~\ref{lem:VG2} we can rewrite \eqref{equ:coro35} as
		\[{
		\begin{aligned}
			&\boldsymbol{U}_2 \cdot \boldsymbol{n} + \frac{\alpha^R \sqrt{A_1 B_1} \boldsymbol{F} + \alpha^L\alpha^R \left( \boldsymbol{U}_2 - \boldsymbol{U}_1 \right)}{\alpha^R - \alpha^L} \cdot \boldsymbol{n} \\
			=& \left( 1 + \frac{\alpha^L \alpha^R}{\alpha^R - \alpha^L}\lambda \right) \boldsymbol{U}_2 \cdot \boldsymbol{n} - \frac{\alpha^L \alpha^R}{\alpha^R - \alpha^L} \lambda \boldsymbol{U}_1 \cdot \boldsymbol{n} + \frac{\alpha^R  }{\alpha^R - \alpha^L}\lambda\sqrt{A_1 B_1} \boldsymbol{F} \cdot \boldsymbol{n}\\
			&- \frac{\alpha^L}{\alpha^R - \alpha^L} \lambda \sqrt{A_2 B_2} \boldsymbol{F} \cdot \boldsymbol{n} \\[0.3em]
			>& \left( 1 + \frac{\alpha^L \alpha^R}{\alpha^R - \alpha^L}\lambda \right) \boldsymbol{U}_2 \cdot \boldsymbol{n} - \frac{\alpha^L \alpha^R}{\alpha^R - \alpha^L} \lambda \boldsymbol{U}_1 \cdot \boldsymbol{n} + \frac{\alpha^R \alpha^L}{\alpha^R - \alpha^L} \lambda \boldsymbol{U}_2 \cdot \boldsymbol{n} \\
			&-\frac{\left(\alpha^L\right)^2}{\alpha^R - \alpha^L} \lambda \boldsymbol{U}_2 \cdot \boldsymbol{n} \\[0.3em]
			=& \left( 1 + \alpha^L\lambda\right)  \boldsymbol{U}_2 \cdot \boldsymbol{n}>0,
		\end{aligned}
		}\]
		and similarly
		\[{
		\begin{aligned}
			&\boldsymbol{U}_1 \cdot \boldsymbol{n} - \frac{\alpha^R \sqrt{A_1 B_1} \boldsymbol{F} + \alpha^L\alpha^R \left( \boldsymbol{U}_2 - \boldsymbol{U}_1 \right)}{\alpha^R - \alpha^L} \cdot \boldsymbol{n}  \\[0.3em]
			=& \left( 1 + \frac{\alpha^L \alpha^R}{\alpha^R - \alpha^L}\lambda \right) \boldsymbol{U}_1 \cdot \boldsymbol{n} - \frac{\alpha^L \alpha^R}{\alpha^R - \alpha^L} \lambda \boldsymbol{U}_2 \cdot \boldsymbol{n} - \frac{\alpha^R  }{\alpha^R - \alpha^L}\lambda\sqrt{A_1 B_1} \boldsymbol{F} \cdot \boldsymbol{n}\\
			&+ \frac{\alpha^L}{\alpha^R - \alpha^L} \lambda \sqrt{A_2 B_2} \boldsymbol{F} \cdot \boldsymbol{n} \\[0.3em]
			>& \left( 1 + \frac{\alpha^L \alpha^R}{\alpha^R - \alpha^L}\lambda \right) \boldsymbol{U}_1 \cdot \boldsymbol{n} - \frac{\alpha^L \alpha^R}{\alpha^R - \alpha^L} \lambda \boldsymbol{U}_2 \cdot \boldsymbol{n} - \frac{\left(\alpha^R\right)^2}{\alpha^R - \alpha^L} \lambda \boldsymbol{U}_2 \cdot \boldsymbol{n} \\[0.3em]
			&+\frac{ \alpha^L \alpha^R}{\alpha^R - \alpha^L} \lambda \boldsymbol{U}_2 \cdot \boldsymbol{n}\\
			=& \left( 1 - \alpha^R\lambda\right) \boldsymbol{U}_1 \cdot \boldsymbol{n}>0.
		\end{aligned}
		}\]
	\end{proof}

	\subsubsection*{Weak CP property}
	As shown in \cref{lem:pointCP}, we apply the classical local scaling limiter based on the minimal point in each cell $I_j$ to enforce the CP property. The limited solution $\widetilde{\mathcal{T}}^{0i}_h$ on cell $I_j$ is defined by
	\begin{equation}\label{eq:def_pcplimit}
		\widetilde{\mathcal{T}}^{0i}_h(r)
		= \theta_j\!\left(\mathcal{T}^{0i}_h(r)-\overline{\mathcal{T}^{0i}_h}\right)
		+ \overline{\mathcal{T}^{0i}_h},\qquad i=0,1,
	\end{equation}
	with
	\[
	\theta_j=\min\!\left\{
	\frac{\varepsilon-\big(\overline{\mathcal{T}^{00}_h}-|\overline{\mathcal{T}^{01}_h}|\big)}
	{\,m_j-\big(\overline{\mathcal{T}^{00}_h}-|\overline{\mathcal{T}^{01}_h}|\big)\,}\,,\ 1
	\right\},
	\quad
	\varepsilon=\min\!\left\{10^{-13},\ \overline{\mathcal{T}^{00}_h}-|\overline{\mathcal{T}^{01}_h}|\right\},
	\]
	and
	\[
	m_j=\min_{r\in I_j}\Big\{\mathcal{T}^{00}_h(r)-\big|\mathcal{T}^{01}_h(r)\big|\Big\}.
	\]
	Here, $\overline{\mathcal{T}^{0i}_h}$ denotes the cell average of $\mathcal{T}^{0i}_h$ on $I_j$. For the second-, third-, and fourth-order CPcOEDG methods, the quantity $m_j$ can be obtained analytically by considering the local extrema and cell boundary points, and thus we omit the details for brevity.
	\begin{lemma}[CP after limiting \label{lem:pointCP}]
		After applying the local scaling limiter \eqref{eq:def_pcplimit}, for every cell $I_j$ and for any point $r \in I_j$, we have 
		\[
		\text{if }  \overline{\mathcal{T}^{00}_h}-|\overline{\mathcal{T}^{01}_h}|\geq \epsilon>0 \text{, then }
		\widetilde{\mathcal{T}}^{00}_h(r)-\big|\widetilde{\mathcal{T}}^{01}_h(r)\big|\;\ge\;\varepsilon\;>\;0.
		\]
	\end{lemma}
	\begin{proof}
		By definition of the limiter and the triangle inequality,
\begin{align*}
\widetilde{\mathcal{T}}^{00}_h(r)-\big|\widetilde{\mathcal{T}}^{01}_h(r)\big|
&\geq \theta_j\!\left(\mathcal{T}^{00}_h(r)-\overline{\mathcal{T}^{00}_h}\right)
  + \overline{\mathcal{T}^{00}_h}
  - \Bigl|\theta_j\!\left(\mathcal{T}^{01}_h(r)-\overline{\mathcal{T}^{01}_h}\right)
  + \overline{\mathcal{T}^{01}_h}\Bigr|\\
&\ge \theta_j
  \Bigl[
    \bigl(\mathcal{T}^{00}_h(r)-|\mathcal{T}^{01}_h(r)|\bigr)
    - \bigl(\overline{\mathcal{T}^{00}_h}-|\overline{\mathcal{T}^{01}_h}|\bigr)
  \Bigr]
  + \bigl(\overline{\mathcal{T}^{00}_h}-|\overline{\mathcal{T}^{01}_h}|\bigr)\\
&\geq \epsilon>0.
\end{align*}

		Hence the stated CP property holds at every point $r \in I_j$.
	\end{proof}
	
	\begin{theorem}[Weak CP]\label{thm:weakCP}
		Assume that for every cell \(I_j\), the cell average \(\overline{\boldsymbol U}_h\), all nodal values \(\boldsymbol U_h(r_{j,l})\), and the auxiliary state $\boldsymbol{U}_{j,*}$ defined in \eqref{eq:Ustar} are admissible, and that the node set \(\{r_{j,l}\}\) contains the interfaces \(\{r_{j\pm\tfrac12}\}\) and all Gauss quadrature nodes \(\{r^{\mathrm{GQ}}_{j,l}\}\).  
		Let \(\mu^{\mathrm{DG}}>0\) and \(\mu^{\mathrm{loc}}>0\) be the method-dependent constants determined by the RK discretization, and choose any \(\delta\in(0,1)\) such that the source contribution at the cell-average level is admissible:
		\[
		\widetilde{\boldsymbol{U}}_{\boldsymbol{S}}(\delta):=\overline{\boldsymbol U}_h+\frac{\Delta t\,\mu^{\mathrm{DG/loc}}}{1-\delta}\,
		\overline{\boldsymbol S(\boldsymbol U_h,A_h,B_h,r)}\in G_c .
		\]
		Let \(\lambda:=\Delta t/\Delta x\). If the CFL condition 
		{\[
		\max\{\alpha_{j\pm\frac12}^R,\alpha_{j\pm\frac12}^L\}\,\tilde{\lambda}^{\,\mathrm{DG/loc}}_{j\pm\frac12}<1
		\]}
		with
		\[
		\tilde{\lambda}^{\,\mathrm{DG/loc}}_{j-\frac12}
		:=\frac{\lambda\,\mu^{\mathrm{DG/loc}}}{\delta\,\gamma_{j,1}},\qquad
		\tilde{\lambda}^{\,\mathrm{DG/loc}}_{j+\frac12}
		:=\frac{\lambda\,\mu^{\mathrm{DG/loc}}}{\delta\,\gamma_{j,N_{\mathrm{GL}}}},
		\]
		holds for all $j$, then the forward-Euler cell-average updates satisfy
		\begin{equation}\label{eq:defUS}
			\overline{\boldsymbol{U}}^{\mathrm{FE,DG}}_{j}
			=\overline{\boldsymbol{U}}_{h}
			+\lambda\,\mu^{\mathrm{DG}}\,
			\mathcal{L}^{\mathrm{DG}}_{j}\big(\boldsymbol{U}_h,A_h,B_h;\boldsymbol{1}\big)\in G_c,
		\end{equation}
		and
		\[
		\overline{\boldsymbol{U}}^{\mathrm{FE},\pm\mathrm{loc}}_{j}
		=\overline{\boldsymbol{U}}_{h}
		\pm \lambda\,\mu^{\mathrm{loc}}\,
		\mathcal{L}^{\mathrm{loc}}_{j}\big(\boldsymbol{U}_h,A_h,B_h;\boldsymbol{1}\big)\in G_c .
		\]
		Here \(\gamma_{j,l}\) and \(r_{j,l}^{\mathrm{GL}}\) are the GL weights and nodes on \(I_j\) with \(\sum_l\gamma_{j,l}=1\), and the overline denotes the cell average on \(I_j\).
	\end{theorem}
	
	\begin{proof}
		We split the cell average into \(\delta\,\overline{\boldsymbol U}_h\) and \((1-\delta)\,\overline{\boldsymbol U}_h\), represent the \(\delta\) portion via GL quadrature, and the \((1-\delta)\) portion via Gauss quadrature. This yields the convex decomposition
		\begin{align*}
			\overline{\boldsymbol{U}}_{h}^{\mathrm{FE},\,\mathrm{DG}/\pm\mathrm{loc}}
			&= \delta\,\gamma_{j,1}\,\boldsymbol{H}_{j,1}^{\mathrm{DG}/\pm\mathrm{loc}}
			+ \delta\,\gamma_{j,N_{\mathrm{GL}}}\,\boldsymbol{H}_{j,N_{\mathrm{GL}}}^{\mathrm{DG}/\pm\mathrm{loc}}
			+ \delta\bigl(1-\gamma_{j,1}-\gamma_{j,N_{\mathrm{GL}}}\bigr)\,	\boldsymbol{U}_{j,*}^{\mathrm{DG}/\pm\mathrm{loc}} \\[-0.2em]
			&\quad + (1-\delta)\widetilde{\boldsymbol{U}}_{\boldsymbol{S}}(\delta),
		\end{align*}
		with endpoint and interior contributions
		\begin{align}
			\boldsymbol{H}_{j,1}^{\mathrm{DG}/\pm\mathrm{loc}}
			&= \boldsymbol{U}_{j,1}^{\mathrm{GL}}
			+ \tilde{\lambda}^{\,\mathrm{DG/loc}}_{j-\frac12}\,
			\widehat{\boldsymbol F}^{\mathrm{DG}/\pm\mathrm{loc}}
			\bigl(\boldsymbol U_{j-\frac12}^{-},\,\boldsymbol U_{j,1}^{\mathrm{GL}}\bigr),\\
			\boldsymbol{H}_{j,N_{\mathrm{GL}}}^{\mathrm{DG}/\pm\mathrm{loc}}
			&= \boldsymbol{U}_{j,N_{\mathrm{GL}}}^{\mathrm{GL}}
			- \tilde{\lambda}^{\,\mathrm{DG/loc}}_{j+\frac12}\,
			\widehat{\boldsymbol F}^{\mathrm{DG}/\pm\mathrm{loc}}
			\bigl(\boldsymbol U_{j,N_{\mathrm{GL}}}^{\mathrm{GL}},\,\boldsymbol U_{j+\frac12}^{+}\bigr),\\
			\boldsymbol{U}_{j,*}^{\text{DG}/\pm\text{loc}} &= \frac{\overline{\boldsymbol{U}}_{h} - \gamma_{j,1}\,\boldsymbol{U}_{j,1}^{\mathrm{GL}} - \gamma_{j,N_{\mathrm{GL}}}\,\boldsymbol{U}_{j,N_{\mathrm{GL}}}^{\mathrm{GL}} }{1-\gamma_{j,1}-\gamma_{j,N_{\mathrm{GL}}}} \in G_c \quad \text{(by the limiting assumption)}.\label{eq:Ustar}
		\end{align}
		Here \(\widehat{\boldsymbol F}^{\mathrm{DG}}\) is the HLL flux, and \(\widehat{\boldsymbol F}^{\pm\mathrm{loc}}\) is the trace of \(\pm\boldsymbol F\) taken from within \(I_j\).
		
		By \cref{lem:VG2} and \cref{coro:LFsplit}, the endpoint states
		\(\boldsymbol{H}_{j,1}^{\mathrm{DG}/\pm\mathrm{loc}}\) and
		\(\boldsymbol{H}_{j,N_{\mathrm{GL}}}^{\mathrm{DG}/\pm\mathrm{loc}}\) are admissible under the condition
	{	\(\max\{\alpha_{j\pm\frac12}^R,\alpha_{j\pm\frac12}^L\}\,\tilde{\lambda}^{\,\mathrm{DG/loc}}_{j\pm\frac12}<1\)}.
		The source contribution is admissible by the choice of \(\delta\).
		Since the weights in the decomposition are nonnegative and sum to one, 
		\(\overline{\boldsymbol{U}}_{h}^{\mathrm{FE},\,\mathrm{DG}/\pm\mathrm{loc}}\in G_c\) by the convexity of $G_c$.
	\end{proof}
	For the admissibility of $\overline{\boldsymbol{U}}_{\boldsymbol{S}}$ in the weak CP property, we have the following lemma, which follows directly from the convexity of $G_c$ and is stated without proof.
	\begin{lemma}\label{lem:USCP}
		If $\overline{\boldsymbol{S}(\boldsymbol{U}_h,A_h,B_h,r)}\in G_c$, we may take $\delta = 1$.
		Otherwise, there exists a unique positive root $\lambda_S$ such that 
		\[
		\mathcal{T}^{00}(\overline{\boldsymbol{U}}_{\boldsymbol{S}}(\lambda_S)) = |\mathcal{T}^{01}(\overline{\boldsymbol{U}}_{\boldsymbol{S}}(\lambda_S))|.
		\]
		For any $\frac{\Delta t \mu^{\mathrm{DG/loc}}}{1-\delta}\in [0,\lambda_S^{-1})$, the convexity of $G_c$ implies 
		\[
		\widetilde{\boldsymbol{U}}_{\boldsymbol{S}}(\delta)\in G_c. 
		\]
	\end{lemma}
	Combining \cref{thm:weakCP} and \cref{lem:USCP}, we can tune $\delta$ to maximize the time step size subject to
	\[
	\frac{\Delta t}{\lambda_S+\max_j\alpha_{j\pm\frac12}\,\tilde{\lambda}^{\,\mathrm{DG/loc}}_{j\pm\frac12}/\Delta t}\leq 1.
	\]
	
	\subsubsection*{Stage cell-average evolution splitting}
	In \cite{LiuSunZhang2025}, the authors show how to write each stage update as a convex combination of forward-Euler steps applied to $\mathcal{L}_j^{\mathrm{DG}}$ and $\mathcal{L}_j^{\pm\mathrm{loc}}$, namely
	\begin{equation}\label{eq:deltat}
		\begin{aligned}
			\overline{\boldsymbol{U}}_h^{n,l+1} &= \sum_{m=0}^l a_{lm}^{\mathrm{DG}}\Bigl(\overline{\boldsymbol{U}}_h^{n,m}+\Delta t\mu_{lm}^{\mathrm{DG}} \mathcal{L}_j^{\mathrm{DG}}({\boldsymbol{U}}_h^{n,m},A_h,B_h,r;\boldsymbol{1})\Bigr)
			\\&\quad+ \sum_{m=0}^la_{lm}^{\mathrm{loc}}\Bigl(\overline{\boldsymbol{U}}_h^{n,m}+\Delta t \mu_{lm}^{\mathrm{loc}} \mathcal{L}_j^{\mathrm{loc}}({\boldsymbol{U}}_h^{n,m},A_h,B_h,r;\boldsymbol{1})\Bigr) 
			\\&\quad+ \sum_{m=0}^l a_{lm}^{\mathrm{-loc}}\Bigl(\overline{\boldsymbol{U}}_h^{n,m}+\Delta t \mu_{lm}^{\mathrm{-loc}} \mathcal{L}_j^{\mathrm{-loc}}({\boldsymbol{U}}_h^{n,m},A_h,B_h,r;\boldsymbol{1})\Bigr),
	\end{aligned}\end{equation}
	with $\sum_{m}(a_{lm}^{\mathrm{DG}}+a_{lm}^{\mathrm{loc}}+a_{lm}^{\mathrm{-loc}}) = 1$. After applying the local scaling limiter, we obtain $\overline{\boldsymbol{U}}_h^{n,l+1}\in G_c$ provided that
	\[
	\frac{\Delta t}{\lambda_S(\boldsymbol{U}_h^{n,m})+\max_{j}\alpha_{j\pm\frac12}\,{{\tilde{\lambda}^{\,\mathrm{DG/loc}}_{j\pm\frac12}(\mu_{lm}^{\mathrm{DG/\pm loc}},\boldsymbol{U}_h^{n,m})}}/\Delta t}\leq 1
	\]
	holds for all $m$. We give concrete examples for the $k$th-order CPcOEDG method with $k = 2,3,4$ in \cref{sec:experiments}, together with the corresponding stage average splittings computed in \cite{LiuSunZhang2025}.
	
	\begin{example}[Second-order CPcOEDG method]\label{ex:2rdCPcOEDG} Following \cite{LiuSunZhang2025,ChenSunXing2024}, 
		the second-order CPcOEDG method can be written as 
		\begin{align*}
			\boldsymbol{U}_\sigma^{n} &= \mathcal{B}\mathcal{F}_{\tau}\boldsymbol{U}_h^n,\\
			\int_{I_j} \boldsymbol{U}_h^{n,1}\!\cdot\!\boldsymbol{\phi}\,{\mathrm{d}r}
			&= \int_{I_j} \boldsymbol{U}_\sigma^{n}\!\cdot\!\boldsymbol{\phi}\,{\mathrm{d}r}
			+ \frac{\Delta t}{2}\,
			\mathcal{L}_j^{\mathrm{loc}}\!\left(\boldsymbol{U}_\sigma^{n,0},A_h^{n,0},B_h^{n,0};\boldsymbol{\phi}\right),\\
			\boldsymbol{U}_\sigma^{n,1}& = \mathcal{B}\boldsymbol{U}_h^{n,1}, \\[0.2em]
			\int_{I_j} \boldsymbol{U}_h^{n+1}\!\cdot\!\boldsymbol{\phi}\,{\mathrm{d}r}
			&= \int_{I_j} \boldsymbol{U}_\sigma^{n}\!\cdot\!\boldsymbol{\phi}\,{\mathrm{d}r} + \Delta t \,
			\mathcal{L}_j^{\mathrm{DG}}\!\left(\boldsymbol{U}_\sigma^{n,1},A_h^{n,1},B_h^{n,1};\boldsymbol{\phi}\right).
		\end{align*}
		Since the first stage is simply a forward-Euler step with the local operator $\mathcal{L}_h^{\mathrm{loc}}$, we only detail the splitting of the last-stage cell average:
		\begin{align*}
			\overline{\boldsymbol{U}}_h^{n+1} =& (1-C)\Bigl[\overline{\boldsymbol{U}}_\sigma^n+\Delta t \frac{C}{2(1-C)}\mathcal{L}_j^{-\mathrm{loc}}(\boldsymbol{U}_{\sigma}^n,A_h^n,B_h^n;\boldsymbol{1})\Bigr]\\
			&\quad+ C\Bigl[ \overline{\boldsymbol{U}}_\sigma^{n,1}+\Delta t\frac{1}{C}\mathcal{L}_j^{\mathrm{DG}}(\boldsymbol{U}_{\sigma}^{n,1},A_h^{n,1},B_h^{n,1};\boldsymbol{1}) \Bigr],\qquad C =\sqrt{3}-1. 
		\end{align*}
	\end{example}
	
	\begin{example}[Third-order CPcOEDG scheme]\label{ex:3rdCPcOEDG} Following \cite{LiuSunZhang2025,ChenSunXing2024}, 
		the third-order CPcOEDG method is given by 
		\begin{align*}
			\boldsymbol{U}_\sigma^{n} &= \mathcal{B}\,\mathcal{F}_{\tau}\,\boldsymbol{U}_h^n,\\[2pt]
			\int_{I_j} \boldsymbol{U}_h^{n,1}\!\cdot\!\boldsymbol{\phi}\,{\mathrm{d}r}
			&= \int_{I_j} \boldsymbol{U}_\sigma^{n}\!\cdot\!\boldsymbol{\phi}\,{\mathrm{d}r}
			+ \frac{\Delta t}{3}\,
			\mathcal{L}_j^{\mathrm{loc}}\!\bigl(\boldsymbol{U}_\sigma^{n,0},A_h^{n,0},B_h^{n,0};\boldsymbol{\phi}\bigr),\\[2pt]
			\int_{I_j} \boldsymbol{U}_h^{n,2}\!\cdot\!\boldsymbol{\phi}\,{\mathrm{d}r}
			&= \int_{I_j} \boldsymbol{U}_\sigma^{n}\!\cdot\!\boldsymbol{\phi}\,{\mathrm{d}r}
			+ \frac{2\Delta t}{3}\,
			\mathcal{L}_j^{\mathrm{loc}}\!\bigl(\mathcal{B}\boldsymbol{U}_h^{n,1},A_h^{n,1},B_h^{n,1};\boldsymbol{\phi}\bigr),\\[2pt]
			\int_{I_j} \boldsymbol{U}_h^{n+1}\!\cdot\!\boldsymbol{\phi}\,{\mathrm{d}r}
			&= \int_{I_j} \boldsymbol{U}_\sigma^{n}\!\cdot\!\boldsymbol{\phi}\,{\mathrm{d}r}
			+ \frac{\Delta t}{4}\,
			\mathcal{L}_j^{\mathrm{DG}}\!\bigl(\boldsymbol{U}_\sigma^{n,0},A_h^{n,0},B_h^{n,0};\boldsymbol{\phi}\bigr)\\
			&\quad+ \frac{3\Delta t}{4}\,
			\mathcal{L}_j^{\mathrm{DG}}\!\bigl(\mathcal{B}\boldsymbol{U}_h^{n,2},A_h^{n,2},B_h^{n,2};\boldsymbol{\phi}\bigr),
		\end{align*}
		where \(\boldsymbol{U}_\sigma^{n,0}\equiv \boldsymbol{U}_\sigma^{n}\). The corresponding cell-average updates admit the convex forms
		\begin{align*}
			\overline{\boldsymbol{U}}_h^{\,n,2}
			&=(1-C_1)\Bigl[\overline{\boldsymbol{U}}_h^{\,n}
			+ \Delta t\,\frac{C_1}{3(1-C_1)}\,
			\mathcal{L}_j^{\mathrm{loc}}\!\bigl(\mathcal{B}\boldsymbol{U}_h^{\,n},A_h^{n},B_h^{n};\boldsymbol{1}\bigr)\Bigr]\\[-1pt]
			&\quad+ C_1\Bigl[\overline{\mathcal{B}\boldsymbol{U}}_h^{\,n,1}
			+ \Delta t\,\frac{2}{3C_1}\,
			\mathcal{L}_j^{\mathrm{loc}}\!\bigl(\mathcal{B}\boldsymbol{U}_h^{\,n,1},A_h^{n,1},B_h^{n,1};\boldsymbol{1}\bigr)\Bigr],
		\end{align*}
		and
		\begin{equation*}
			\resizebox{0.99\hsize}{!}{$\begin{aligned}
					\overline{\boldsymbol{U}}_h^{\,n+1}
					&=(1-C_2-C_3)\Bigl[\hat C\,\overline{\boldsymbol{U}}_h^{\,n}
					+ \Delta t\,\frac{C_2}{3(1-C_2-C_3)\hat C}\,
					\mathcal{L}_j^{\mathrm{loc}}\!\bigl(\mathcal{B}\boldsymbol{U}_h^{\,n},A_h^{n},B_h^{n};\boldsymbol{1}\bigr)\Bigr]\\
					&\quad+(1-C_2-C_3)\Bigl[(1-\hat C)\,\overline{\boldsymbol{U}}_h^{\,n}
					+ \Delta t\,\frac{1}{4(1-C_2-C_3)(1-\hat C)}\,
					\mathcal{L}_j^{\mathrm{DG}}\!\bigl(\mathcal{B}\boldsymbol{U}_h^{\,n},A_h^{n},B_h^{n};\boldsymbol{1}\bigr)\Bigr]\\
					&\quad+ C_2\Bigl[\overline{\mathcal{B}\boldsymbol{U}}_h^{\,n,1}
					+ \Delta t\,\frac{2C_3}{3C_2}\,
					\mathcal{L}_j^{\mathrm{loc}}\!\bigl(\mathcal{B}\boldsymbol{U}_h^{\,n,1},A_h^{n,1},B_h^{n,1};\boldsymbol{1}\bigr)\Bigr]\\
					&\quad+ C_3\Bigl[\overline{\mathcal{B}\boldsymbol{U}}_h^{\,n,2}
					+ \Delta t\,\frac{3}{4C_3}\,
					\mathcal{L}_j^{\mathrm{DG}}\!\bigl(\mathcal{B}\boldsymbol{U}_h^{\,n,2},A_h^{n,2},B_h^{n,2};\boldsymbol{1}\bigr)\Bigr],\end{aligned}$}
		\end{equation*}
		where \(\overline{(\cdot)}\) denotes the cell average over \(I_j\), 
		\((C_1,C_2,C_3)=(0.4764,\;0.2442,\;0.5242)\), and \(\hat C\in[0,1]\) is a free parameter chosen to \emph{minimize}
		\[
		\max\!\left\{\,
		\frac{\lambda^{\mathrm{loc}}\,C_2}{3(1-C_2-C_3)\,\hat C}\;,\;
		\frac{\lambda^{\mathrm{DG}}}{4(1-C_2-C_3)\,(1-\hat C)}
		\,\right\},
		\]
		thereby maximizing the admissible time step. Here \(\lambda^{\mathrm{loc}}\) and \(\lambda^{\mathrm{DG}}\) are the effective CFL prefactors associated with the local and DG operators, respectively.
	\end{example}
	
	\begin{example}[Fourth-order CPcOEDG method]\label{ex:4thCPcOEDG}
		A key advantage of the cRKDG framework is that it bypasses the SSP order barrier: no four-stage, fourth-order SSP RK scheme exists. The fourth-order CPcOEDG method is given in  \cite{LiuSunZhang2025,ChenSunXing2024} and also presented in \cref{app:RK4} for completeness. The stage cell averages
		$\overline{\boldsymbol{U}}_h^{\,n,\ell}$ for $\ell=2,3,4$ split as
		\begin{align*}
			\overline{\boldsymbol{U}}_h^{\,n,2}
			&=(1-C_1)\Bigl[\overline{\boldsymbol{U}}_h^{\,n}
			+ \Delta t\,\frac{C_1}{2(1-C_1)}\,
			\mathcal{L}_j^{-\mathrm{loc}}\!\bigl(\mathcal{B}\boldsymbol{U}_h^{\,n},A_h^{n},B_h^{n};\boldsymbol{1}\bigr)\Bigr]\\
			&\quad+ C_1\Bigl[\overline{\mathcal{B}\boldsymbol{U}}_h^{\,n,1}
			+ \Delta t\,\frac{1}{2C_1}\,
			\mathcal{L}_j^{\mathrm{loc}}\!\bigl(\mathcal{B}\boldsymbol{U}_h^{\,n,1},A_h^{n,1},B_h^{n,1};\boldsymbol{1}\bigr)\Bigr],\\[4pt]
			\overline{\boldsymbol{U}}_h^{\,n,3}
			&=(1-C_2-C_3)\Bigl[\overline{\boldsymbol{U}}_h^{\,n}
			+ \Delta t\,\frac{C_2}{2(1-C_2-C_3)}\,
			\mathcal{L}_j^{-\mathrm{loc}}\!\bigl(\mathcal{B}\boldsymbol{U}_h^{\,n},A_h^{n},B_h^{n};\boldsymbol{1}\bigr)\Bigr]\\
			&\quad+ C_2\Bigl[\overline{\mathcal{B}\boldsymbol{U}}_h^{\,n,1}
			+ \Delta t\,\frac{C_3}{2C_2}\,
			\mathcal{L}_j^{-\mathrm{loc}}\!\bigl(\mathcal{B}\boldsymbol{U}_h^{\,n,1},A_h^{n,1},B_h^{n,1};\boldsymbol{1}\bigr)\Bigr]\\
			&\quad+ C_3\Bigl[\overline{\mathcal{B}\boldsymbol{U}}_h^{\,n,2}
			+ \Delta t\,\frac{1}{C_3}\,
			\mathcal{L}_j^{-\mathrm{loc}}\!\bigl(\mathcal{B}\boldsymbol{U}_h^{\,n,2},A_h^{n,2},B_h^{n,2};\boldsymbol{1}\bigr)\Bigr].
		\end{align*}
		
		The final update can also be written in convex form \cite{LiuSunZhang2025}:
		\begin{equation*}
			\resizebox{0.9999\hsize}{!}{$\begin{aligned}
					\overline{\boldsymbol{U}}_h^{\,n+1}
					&=(1-C_4-C_5-C_6)\Bigl[\overline{\boldsymbol{U}}_h^{\,n}
					+ \Delta t\,\frac{C_4}{2(1-C_4-C_5-C_6)}\,
					\mathcal{L}_j^{-\mathrm{loc}}\!\bigl(\mathcal{B}\boldsymbol{U}_h^{\,n},A_h^{n},B_h^{n};\boldsymbol{1}\bigr)\\[-0.2em]
					&\qquad+ \Delta t\,\frac{1}{6(1-C_4-C_5-C_6)}\,
					\mathcal{L}_j^{\mathrm{DG}}\!\bigl(\mathcal{B}\boldsymbol{U}_h^{\,n},A_h^{n},B_h^{n};\boldsymbol{1}\bigr)\Bigr]\\[0.2em]
					&\quad+ C_4\Bigl[\overline{\mathcal{B}\boldsymbol{U}}_h^{\,n,1}
					+ \Delta t\,\frac{C_5}{2C_4}\,
					\mathcal{L}_j^{-\mathrm{loc}}\!\bigl(\mathcal{B}\boldsymbol{U}_h^{\,n,1},A_h^{n,1},B_h^{n,1};\boldsymbol{1}\bigr)
					+ \Delta t\,\frac{1}{3C_4}\,
					\mathcal{L}_j^{\mathrm{DG}}\!\bigl(\mathcal{B}\boldsymbol{U}_h^{\,n,1},A_h^{n,1},B_h^{n,1};\boldsymbol{1}\bigr)\Bigr]\\
					&\quad+ C_5\Bigl[\overline{\mathcal{B}\boldsymbol{U}}_h^{\,n,2}
					+ \Delta t\,\mathcal{L}_j^{-\mathrm{loc}}\!\bigl(\mathcal{B}\boldsymbol{U}_h^{\,n,2},A_h^{n,2},B_h^{n,2};\boldsymbol{1}\bigr)
					+ \Delta t\,\frac{1}{3C_5}\,
					\mathcal{L}_j^{\mathrm{DG}}\!\bigl(\mathcal{B}\boldsymbol{U}_h^{\,n,2},A_h^{n,2},B_h^{n,2};\boldsymbol{1}\bigr)\Bigr]\\
					&\quad+ C_6\Bigl[\overline{\mathcal{B}\boldsymbol{U}}_h^{\,n,3}
					+ \Delta t\,\frac{1}{6C_6}\,
					\mathcal{L}_j^{\mathrm{DG}}\!\bigl(\mathcal{B}\boldsymbol{U}_h^{\,n,3},A_h^{n,3},B_h^{n,3};\boldsymbol{1}\bigr)\Bigr].\end{aligned}$}
		\end{equation*}
		Here $(C_1,\dots,C_6)=(0.5,\,0.2346,\,0.6850,\,0.3334,\,0.3066,\,0.1142)$.
	\end{example}

	\section{Numerical Experiments}\label{sec:experiments}
	In this section, we validate the performance of the proposed CPcOEDG method through a series of benchmark tests for the spherically symmetric EE system. We employ a barotropic equation of state of the form	
	\begin{equation}\label{eq:rho2p}
		p=\tilde{\sigma}^{2}\rho,\qquad \tilde{\sigma}\in(0,1).
	\end{equation} 
	The tests cover several representative curved spacetimes, including the Schwarzschild, Tolman--Oppenheimer--Volkoff (TOV), and Friedmann--Robertson--Walker (FRW) metrics.
	Unless otherwise stated, we set $\kappa = 8\pi$. While the theoretical CFL condition derived in \eqref{eq:deltat} provides a sufficient condition for constraint preservation, in practice somewhat larger time steps can often be used. In our computations, we adopt CFL numbers of 0.317, 0.169, and 0.05 for the second-, third-, and fourth-order schemes, respectively. 
	
	For time-dependent Dirichlet boundary conditions, prescribing boundary data at intermediate RK stage times \(g\!\left(t_{n}^{l}\right)\) is known to reduce accuracy to second order for general nonlinear systems \cite{Gustafsson1975}. We therefore employ the RK boundary update strategy for time-dependent Dirichlet data proposed in \cite{CarpenterGottliebAbarbanelDon1995}, which in our experiments maintains the intended high-order accuracy. Overall, the tests confirm that the proposed CPcOEDG method is both high-order accurate and effective. 
	
	All computations are implemented in \texttt{C++} with \texttt{long double} precision. For several examples, reference solutions are obtained with a Godunov scheme as in \cite{WuTang2016}.

	\subsubsection*{Boundary treatment for Dirichlet problems}
	For Heun's third-order RK scheme, the stage boundary data are
	\begin{align}
		& b_0^1 = g(t) + \frac{\tau}{3} g'(t), \nonumber\\
		& b_0^2 = g(t) + \frac{2\tau}{3} g'(t) + \frac{2\tau^2}{9} g''(t), \nonumber
	\end{align}
	and for the classical fourth-order RK method,
	\begin{align}
		& b_0^1 = g(t) + \frac{\tau}{2} g'(t), \nonumber\\
		& b_0^2 = g(t) + \frac{\tau}{2} g'(t) + \frac{\tau^2}{4} g''(t), \nonumber\\
		& b_0^3 = g(t) + \tau g'(t) + \frac{\tau^2}{2} g''(t) + \frac{\tau^3}{4} g'''(t).\nonumber
	\end{align}
	Here, $b_0^l$ denotes the $l$th stage approximation of the boundary data corresponding to the exact boundary function $g(t)$.  
	
	\subsection{Accuracy tests}
	We first assess the convergence rates of the proposed CPcOEDG scheme under Dirichlet boundary conditions prescribed by exact solutions. Three fully relativistic smooth models are considered, including two distinct coordinate transformations of the flat FRW metric 
	\begin{equation}\label{eq:FRW}
		\mathrm{d}s^2=-\,\mathrm{d}\tilde t^{\,2}+\tilde t\Big(\mathrm{d}\tilde r^{\,2}+\tilde r^{\,2}\big(\mathrm{d}\theta^{2}+\sin^{2}\theta\,\mathrm{d}\phi^{2}\big)\Big),\quad \tilde t \text{ measured from the Big Bang},
	\end{equation}
	and the TOV metric \cite{Vogler2010}. For each case we report the $L^1$, $L^2$, and $L^\infty$ errors and convergence rates of $\rho$, $v$, $A$, and $B$ obtained by the second-, third-, and fourth-order CPcOEDG methods.

	\begin{exmp}[FRW-1 model]\label{ex:FRW1}
		For the first FRW test, we adopt the coordinate transformation 
		\begin{equation}
			t=\tilde{t}+\frac{\tilde{r}^2}{4}, \quad r=\tilde{r} \sqrt{\tilde{t}}.
		\end{equation}
		This leads to the Schwarzschild-form metric \eqref{eq:metric} with
		\begin{equation}\label{eq:FRW1metric}
			A(t, r)=1-v^2, \quad B(t, r)=\frac{1}{1-v^2}.
		\end{equation}
		The exact solution is
		\begin{equation}\label{sol:FRW1}
			\rho(t, r)=\frac{16 v^2}{3\left(1+\tilde{\sigma}^2\right)^2 \kappa r^2}, \quad v(t, r)=\frac{1-\sqrt{1-\xi^2}}{\xi},\quad \text{with } \tilde{\sigma} = \frac{1}{\sqrt{3}},
		\end{equation}
		where $\xi := \frac{r}{t}$. 
		
		The problem is solved on uniform partitions of the domain \(\Omega=[3,7]\) over \(t\in[15,16]\) using second-, third-, and fourth-order CPcOEDG schemes. Boundary values are prescribed from the exact solution \eqref{sol:FRW1}. Results at $t=16$ for the \(\mathbb{P}^1\), \(\mathbb{P}^2\), and \(\mathbb{P}^3\) CPcOEDG schemes are reported in \cref{tab:TW1P1}--\cref{tab:TW1P3} and confirm high-order accuracy with the expected optimal rates. 
		
		As noted in \cref{rem:OEDGandcOEDG}, we also compare the density errors of the fourth-order CPcOEDG and OEDG schemes in \cref{tab:TW1P3OEDG}. The CPcOEDG method attains the optimal rate, whereas the OEDG scheme exhibits a reduction from the expected fourth order to below third order in the \(L^\infty\) norm.
		
		\begin{table}[htp]\label{tab:TW1P1}
			\centering
			\belowrulesep=0pt
			\aboverulesep=0pt
			\caption{ Errors and convergence rates of the second-order CPcOEDG scheme for \cref{ex:FRW1}.}
			
			\setlength{\tabcolsep}{3mm}
			\begin{tabular}{c|cccccccc}
				\toprule[1.5pt]
				\multirow{2}{*}{} &
				\multirow{2}{*}{$N$} &
				\multicolumn{2}{c}{$L^1$ norm} &
				\multicolumn{2}{c}{$L^2$ norm} &
				\multicolumn{2}{c}{$L^\infty$ norm} \\
				\cmidrule(r){3-4} \cmidrule(r){5-6} \cmidrule(l){7-8}
				& & error & order &  error & order &  error & order \\    
				\midrule[1.5pt]
				\multirow{5}{*}{$\rho$} 
				& 30&	2.52E-09&	--&	1.62E-09&	--&	1.79E-09&	--\\ 
				& 60&	6.01E-10&	2.07&	3.25E-10&	2.32&	3.07E-10&	2.54\\ 
				& 90&	2.64E-10&	2.03&	1.38E-10&	2.12&	1.20E-10&	2.31\\ 
				& 120&	1.47E-10&	2.02&	7.56E-11&	2.08&	6.08E-11&	2.37\\ 
				& 150&	9.40E-11&	2.02&	4.79E-11&	2.04&	3.59E-11&	2.36\\
				\midrule
				\multirow{5}{*}{$v$} 
				& 30&	6.28E-06&	--&	4.29E-06&	--&	4.93E-06&	--\\ 
				& 60&	1.34E-06&	2.23&	8.66E-07&	2.31&	1.05E-06&	2.23\\ 
				& 90&	5.69E-07&	2.12&	3.75E-07&	2.06&	5.04E-07&	1.81\\ 
				& 120&	3.13E-07&	2.08&	2.10E-07&	2.02&	2.95E-07&	1.86\\ 
				& 150&	1.98E-07&	2.06&	1.35E-07&	1.98&	1.94E-07&	1.89\\ 
				\midrule
				\multirow{5}{*}{$A$} 
				& 30&	1.66E-05&	--&	8.40E-06&	--&	5.49E-06&	--\\ 
				& 60&	4.12E-06&	2.01&	2.08E-06&	2.01&	1.36E-06&	2.02\\ 
				& 90&	1.82E-06&	2.01&	9.22E-07&	2.01&	6.01E-07&	2.01\\ 
				& 120&	1.02E-06&	2.00&	5.18E-07&	2.01&	3.37E-07&	2.01\\ 
				& 150&	6.55E-07&	2.00&	3.31E-07&	2.00&	2.16E-07&	2.01\\  
				\midrule
				\multirow{5}{*}{$B$} 
				& 30&	2.88E-05&	--&	1.51E-05&	--&	1.15E-05&	--\\ 
				& 60&	7.17E-06&	2.00&	3.76E-06&	2.00&	2.87E-06&	2.00\\ 
				& 90&	3.19E-06&	2.00&	1.67E-06&	2.00&	1.28E-06&	2.00\\ 
				& 120&	1.79E-06&	2.00&	9.40E-07&	2.00&	7.18E-07&	2.00\\ 
				& 150&	1.15E-06&	2.00&	6.01E-07&	2.00&	4.60E-07&	2.00\\  
				\bottomrule[1.5pt]
			\end{tabular}
		\end{table}
		
		\begin{table}[htp]\label{tab:TW1P2}
			\centering
			\belowrulesep=0pt
			\aboverulesep=0pt
			\caption{Errors and convergence rates of the third-order CPcOEDG method for \cref{ex:FRW1}.}
			
			\setlength{\tabcolsep}{3mm}
			\begin{tabular}{c|cccccccc}
				\toprule[1.5pt]
				\multirow{2}{*}{} &
				\multirow{2}{*}{$N$} &
				\multicolumn{2}{c}{$L^1$ norm} &
				\multicolumn{2}{c}{$L^2$ norm} &
				\multicolumn{2}{c}{$L^\infty$ norm} \\
				\cmidrule(r){3-4} \cmidrule(r){5-6} \cmidrule(l){7-8}
				& & error & order &  error & order &  error & order \\    
				\midrule[1.5pt]
				\multirow{5}{*}{$\rho$} 
				& 30&	8.75E-12&	--&	5.71E-12&	--&	8.62E-12&	--\\ 
				& 60&	1.19E-12&	2.88&	7.49E-13&	2.93&	1.08E-12&	3.00\\ 
				& 90&	3.64E-13&	2.93&	2.25E-13&	2.96&	3.22E-13&	2.98\\ 
				& 120&	1.51E-13&	3.05&	9.29E-14&	3.08&	1.33E-13&	3.09\\ 
				& 150&	7.32E-14&	3.26&	4.52E-14&	3.23&	6.44E-14&	3.23\\  
				\midrule
				\multirow{5}{*}{$v$} 
				& 30&	2.84E-08&	--&	1.80E-08&	--&	2.62E-08&	--\\ 
				& 60&	3.80E-09&	2.90&	2.29E-09&	2.97&	3.37E-09&	2.96\\ 
				& 90&	1.16E-09&	2.93&	6.95E-10&	2.95&	1.01E-09&	2.96\\ 
				& 120&	4.84E-10&	3.04&	2.91E-10&	3.02&	4.19E-10&	3.07\\ 
				& 150&	2.33E-10&	3.27&	1.43E-10&	3.18&	2.02E-10&	3.28\\ 
				\midrule
				\multirow{5}{*}{$A$} 
				& 30&	1.18E-08&	--&	8.23E-09&	--&	9.22E-09&	--\\ 
				& 60&	1.48E-09&	2.99&	1.03E-09&	3.00&	1.17E-09&	2.97\\ 
				& 90&	4.40E-10&	3.00&	3.05E-10&	3.00&	3.49E-10&	2.99\\ 
				& 120&	1.86E-10&	3.00&	1.29E-10&	3.00&	1.48E-10&	2.99\\ 
				& 150&	9.52E-11&	3.00&	6.59E-11&	3.00&	7.58E-11&	2.99\\   
				\midrule
				\multirow{5}{*}{$B$} 
				& 30&	1.99E-08&	--&	1.43E-08&	--&	1.71E-08&	--\\ 
				& 60&	2.48E-09&	3.00&	1.79E-09&	3.00&	2.18E-09&	2.98\\ 
				& 90&	7.35E-10&	3.00&	5.32E-10&	3.00&	6.48E-10&	2.99\\ 
				& 120&	3.10E-10&	3.00&	2.24E-10&	3.00&	2.74E-10&	2.99\\ 
				& 150&	1.59E-10&	3.00&	1.15E-10&	3.00&	1.41E-10&	2.99\\  
				\bottomrule[1.5pt]
			\end{tabular}
		\end{table}
		
		\begin{table}[htp]\label{tab:TW1P3}
			\centering
			\belowrulesep=0pt
			\aboverulesep=0pt
			\caption{Errors and convergence rates of the fourth-order CPcOEDG method for \cref{ex:FRW1}.}
			
			\setlength{\tabcolsep}{3mm}
			\begin{tabular}{c|cccccccc}
				\toprule[1.5pt]
				\multirow{2}{*}{} &
				\multirow{2}{*}{$N$} &
				\multicolumn{2}{c}{$L^1$ norm} &
				\multicolumn{2}{c}{$L^2$ norm} &
				\multicolumn{2}{c}{$L^\infty$ norm} \\
				\cmidrule(r){3-4} \cmidrule(r){5-6} \cmidrule(l){7-8}
				& & error & order &  error & order &  error & order \\    
				\midrule[1.5pt]
				\multirow{5}{*}{$\rho$} 
				& 30&	3.16E-14&	--&	2.46E-14&	--&	4.79E-14&	--\\ 
				& 60&	1.90E-15&	4.06&	1.29E-15&	4.25&	2.43E-15&	4.30\\ 
				& 90&	3.49E-16&	4.18&	2.34E-16&	4.22&	4.35E-16&	4.25\\ 
				& 120&	1.16E-16&	3.83&	7.55E-17&	3.93&	1.37E-16&	4.02\\ 
				& 150&	4.56E-17&	4.18&	2.97E-17&	4.18&	5.37E-17&	4.19\\  
				\midrule
				\multirow{5}{*}{$v$} 
				& 30&	7.35E-11&	--&	4.88E-11&	--&	8.48E-11&	--\\ 
				& 60&	4.96E-12&	3.89&	2.97E-12&	4.04&	5.12E-12&	4.05\\ 
				& 90&	9.65E-13&	4.04&	5.73E-13&	4.06&	1.06E-12&	3.89\\ 
				& 120&	3.16E-13&	3.88&	1.90E-13&	3.83&	3.49E-13&	3.85\\ 
				& 150&	1.28E-13&	4.07&	7.66E-14&	4.07&	1.42E-13&	4.01\\  
				\midrule
				\multirow{5}{*}{$A$} 
				& 30&	4.06E-11&	--&	2.28E-11&	--&	2.86E-11&	--\\ 
				& 60&	2.46E-12&	4.04&	1.38E-12&	4.05&	1.74E-12&	4.04\\ 
				& 90&	4.82E-13&	4.03&	2.69E-13&	4.03&	3.41E-13&	4.02\\ 
				& 120&	1.51E-13&	4.02&	8.46E-14&	4.02&	1.08E-13&	4.01\\ 
				& 150&	6.16E-14&	4.03&	3.45E-14&	4.03&	4.39E-14&	4.02\\   
				\midrule
				\multirow{5}{*}{$B$} 
				& 30&	7.22E-11&	--&	4.15E-11&	--&	5.33E-11&	--\\ 
				& 60&	4.52E-12&	4.00&	2.60E-12&	4.00&	3.34E-12&	4.00\\ 
				& 90&	8.94E-13&	4.00&	5.13E-13&	4.00&	6.61E-13&	4.00\\ 
				& 120&	2.83E-13&	4.00&	1.62E-13&	4.00&	2.09E-13&	4.00\\ 
				& 150&	1.16E-13&	4.00&	6.65E-14&	4.00&	8.58E-14&	4.00\\    
				\bottomrule[1.5pt]
			\end{tabular}
		\end{table}
		
		\begin{table}[htp]\label{tab:TW1P3OEDG}
			\centering
			\belowrulesep=0pt
			\aboverulesep=0pt
			\caption{Errors and convergence rates for the density $\rho$ and velocity $v$ in \cref{ex:FRW1} obtained with the fourth-order OEDG method.}
			\setlength{\tabcolsep}{3mm}
			\begin{tabular}{c|cccccccc}
				\toprule[1.5pt]
				\multirow{2}{*}{} &
				\multirow{2}{*}{$N$} &
				\multicolumn{2}{c}{$L^1$ norm} &
				\multicolumn{2}{c}{$L^2$ norm} &
				\multicolumn{2}{c}{$L^\infty$ norm} \\
				\cmidrule(r){3-4} \cmidrule(r){5-6} \cmidrule(l){7-8}
				& & error & order &  error & order &  error & order \\    
				\midrule[1.5pt]
				\multirow{5}{*}{$\rho$} 
				& 30&	8.83E-11&	--&	5.85E-11&	--&	9.81E-11&	--\\ 
				& 60&	5.98E-12&	3.88&	3.56E-12&	4.04&	5.33E-12&	4.20\\ 
				& 90&	1.19E-12&	3.97&	6.98E-13&	4.02&	1.08E-12&	3.93\\ 
				& 120&	3.81E-13&	3.96&	2.23E-13&	3.97&	3.51E-13&	3.92\\ 
				& 150&	1.57E-13&	3.98&	9.20E-14&	3.96&	2.12E-13&	2.25\\  
				\bottomrule[1.5pt]
			\end{tabular}
		\end{table}
	\end{exmp}
	
	\begin{exmp}[FRW-2 model]\label{ex:FRW2}
		We next consider an alternative coordinate transformation to standard Schwarzschild coordinates,
		\begin{equation}
			t=\frac{\Psi_0}{2} \sqrt{\frac{4 \tilde{t}^2+\tilde{t} \tilde{r}^2}{\tilde{t}}}, \quad r=\tilde{r} \sqrt{\tilde{t}},
		\end{equation}
		applied to the FRW metric \eqref{eq:FRW}. The resulting Schwarzschild metric is again of the form \eqref{eq:metric}, with
		\begin{equation}
			A(t, r)=1-v^2, \quad B(t, r)=\frac{1}{\Psi(t,r)\left(1-v^2\right)},\qquad
			\Psi(t, r)=\Psi_0 \sqrt{\frac{\tilde{t}}{4 \tilde{t}^2+r^2}},
		\end{equation}
		where $\Psi_0 \equiv 1$ in the computations. The constant $\tilde{\sigma}$ in the equation of state is the same as in \eqref{sol:FRW1}. The exact solution is
		\begin{equation}
			(\rho,v)=\left(\frac{4}{3\left(1+\tilde{\sigma}^2\right)^2 \kappa {\tilde{t}}^2},\frac{r}{2 \tilde{t}}\right),\qquad
			\tilde{t}=\frac{t^2+\sqrt{t^4-r^2 \Psi_0^4}}{2 \Psi_0^2}.
		\end{equation}
		
		\cref{tb:3_2,tb:3_3,tb:3_4} display the numerical $L^1$, $L^2$, and $L^\infty$ errors and convergence rates at $t= 16$ of all components $(\rho,v,A,B)$ for the $k$th-order CPcOEDG schemes with $k = 2,3,4$. The results show that the observed convergence rates closely match the theoretical order $k$ for each method and for all three norms, even in the presence of nonhomogeneous Dirichlet boundaries. 
		
		\begin{table}[htp]\label{tb:3_2}
			\centering
			\belowrulesep=0pt
			\aboverulesep=0pt
			\caption{Errors and convergence rates for the second-order CPcOEDG scheme in \cref{ex:FRW2}.}

			\setlength{\tabcolsep}{3mm}
			\begin{tabular}{c|cccccccc}
				\toprule[1.5pt]
				\multirow{2}{*}{} &
				\multirow{2}{*}{$N$} &
				\multicolumn{2}{c}{$L^1$ norm} &
				\multicolumn{2}{c}{$L^2$ norm} &
				\multicolumn{2}{c}{$L^\infty$ norm} \\
				\cmidrule(r){3-4} \cmidrule(r){5-6} \cmidrule(l){7-8}
				& & error & order &  error & order &  error & order \\    
				\midrule[1.5pt]
				\multirow{5}{*}{$\rho$} 
				& 10&	5.00E-13&	--&	2.53E-13&	--&	1.59E-13&	--\\ 
				& 20&	1.25E-13&	2.00&	6.33E-14&	2.00&	4.04E-14&	1.98\\ 
				& 40&	3.12E-14&	2.00&	1.58E-14&	2.00&	1.02E-14&	1.99\\ 
				& 80&	7.81E-15&	2.00&	3.95E-15&	2.00&	2.54E-15&	2.00\\ 
				& 160&	1.95E-15&	2.00&	9.89E-16&	2.00&	6.41E-16&	1.99\\
				\midrule
				\multirow{5}{*}{$v$} 
				& 10&	1.58E-07&	--&	9.12E-08&	--&	8.33E-08&	--\\ 
				& 20&	4.16E-08&	1.93&	2.40E-08&	1.92&	2.36E-08&	1.82\\ 
				& 40&	1.03E-08&	2.01&	5.99E-09&	2.00&	6.26E-09&	1.91\\ 
				& 80&	2.47E-09&	2.06&	1.45E-09&	2.05&	1.58E-09&	1.98\\ 
				& 160&	6.92E-10&	1.83&	3.99E-10&	1.86&	4.26E-10&	1.89\\ 
				\midrule
				\multirow{5}{*}{$A$} 
				& 10&	4.08E-07&	--&	2.04E-07&	--&	1.02E-07&	--\\ 
				& 20&	1.02E-07&	2.00&	5.10E-08&	2.00&	2.55E-08&	2.00\\ 
				& 40&	2.55E-08&	2.00&	1.27E-08&	2.00&	6.37E-09&	2.00\\ 
				& 80&	6.37E-09&	2.00&	3.19E-09&	2.00&	1.59E-09&	2.00\\ 
				& 160&	1.59E-09&	2.00&	7.96E-10&	2.00&	3.98E-10&	2.00\\    
				\midrule
				\multirow{5}{*}{$B$} 
				& 10&	7.19E-04&	--&	3.67E-04&	--&	2.35E-04&	--\\ 
				& 20&	1.80E-04&	2.00&	9.19E-05&	2.00&	5.90E-05&	1.99\\ 
				& 40&	4.50E-05&	2.00&	2.30E-05&	2.00&	1.48E-05&	2.00\\ 
				& 80&	1.12E-05&	2.00&	5.74E-06&	2.00&	3.70E-06&	2.00\\ 
				& 160&	2.81E-06&	2.00&	1.44E-06&	2.00&	9.26E-07&	2.00\\   
				\bottomrule[1.5pt]
			\end{tabular}
		\end{table}
		
		\begin{table}[htp]\label{tb:3_3}
			\centering
			\belowrulesep=0pt
			\aboverulesep=0pt
			\caption{Errors and convergence rates for \cref{ex:FRW2} using the third-order CPcOEDG scheme.}
			
			\setlength{\tabcolsep}{3mm}
			\begin{tabular}{c|cccccccc}
				\toprule[1.5pt]
				\multirow{2}{*}{} &
				\multirow{2}{*}{$N$} &
				\multicolumn{2}{c}{$L^1$ norm} &
				\multicolumn{2}{c}{$L^2$ norm} &
				\multicolumn{2}{c}{$L^\infty$ norm} \\
				\cmidrule(r){3-4} \cmidrule(r){5-6} \cmidrule(l){7-8}
				& & error & order &  error & order &  error & order \\    
				\midrule[1.5pt]
				\multirow{5}{*}{$\rho$} 
				& 4&	7.59E-15&	--&	3.84E-15&	--&	2.22E-15&	--\\ 
				& 6&	2.25E-15&	3.00&	1.14E-15&	3.00&	6.63E-16&	2.98\\ 
				& 8&	9.51E-16&	2.99&	4.80E-16&	3.00&	2.75E-16&	3.06\\ 
				& 10&	4.88E-16&	2.99&	2.45E-16&	3.01&	1.35E-16&	3.18\\ 
				& 12&	2.83E-16&	3.00&	1.42E-16&	3.00&	7.83E-17&	2.99\\ 
				\midrule
				\multirow{5}{*}{$v$} 
				& 4&	1.17E-09&	--&	7.05E-10&	--&	5.94E-10&	--\\ 
				& 6&	3.69E-10&	2.84&	2.24E-10&	2.83&	1.87E-10&	2.85\\ 
				& 8&	1.57E-10&	2.98&	9.54E-11&	2.96&	7.96E-11&	2.97\\ 
				& 10&	7.44E-11&	3.34&	4.53E-11&	3.34&	3.81E-11&	3.30\\ 
				& 12&	4.05E-11&	3.33&	2.46E-11&	3.36&	2.08E-11&	3.33\\ 
				\midrule
				\multirow{5}{*}{$A$} 
				& 4&	5.04E-11&	--&	3.44E-11&	--&	2.95E-11&	--\\ 
				& 6&	1.49E-11&	3.00&	1.02E-11&	3.00&	8.94E-12&	2.95\\ 
				& 8&	6.31E-12&	3.00&	4.31E-12&	3.00&	3.81E-12&	2.96\\ 
				& 10&	3.23E-12&	3.00&	2.21E-12&	3.00&	1.97E-12&	2.96\\ 
				& 12&	1.88E-12&	2.98&	1.28E-12&	3.00&	1.15E-12&	2.95\\   
				\midrule
				\multirow{5}{*}{$B$} 
				& 4&	8.10E-08&	--&	5.30E-08&	--&	4.59E-08&	--\\ 
				& 6&	2.36E-08&	3.04&	1.57E-08&	3.00&	1.39E-08&	2.95\\ 
				& 8&	9.88E-09&	3.03&	6.63E-09&	3.00&	5.92E-09&	2.96\\ 
				& 10&	5.04E-09&	3.02&	3.40E-09&	3.00&	3.05E-09&	2.97\\ 
				& 12&	2.91E-09&	3.02&	1.97E-09&	3.00&	1.77E-09&	2.98\\ 
				\bottomrule[1.5pt]
			\end{tabular}
		\end{table}
		
		\begin{table}[htp]\label{tb:3_4}
			\centering
			\belowrulesep=0pt
			\aboverulesep=0pt
			\caption{Errors and convergence rates for \cref{ex:FRW2} using the fourth-order CPcOEDG method.}
			
			\setlength{\tabcolsep}{3mm}
			\begin{tabular}{c|cccccccc}
				\toprule[1.5pt]
				\multirow{2}{*}{} &
				\multirow{2}{*}{$N$} &
				\multicolumn{2}{c}{$L^1$ norm} &
				\multicolumn{2}{c}{$L^2$ norm} &
				\multicolumn{2}{c}{$L^\infty$ norm} \\
				\cmidrule(r){3-4} \cmidrule(r){5-6} \cmidrule(l){7-8}
				& & error & order &  error & order &  error & order \\    
				\midrule[1.5pt]
				\multirow{5}{*}{$\rho$} 
				& 4&	3.76E-18&	--&	1.89E-18&	--&	1.06E-18&	--\\ 
				& 6&	7.46E-19&	3.99&	3.75E-19&	3.99&	2.11E-19&	3.98\\ 
				& 8&	2.37E-19&	3.99&	1.19E-19&	3.99&	6.71E-20&	3.97\\ 
				& 10&	9.72E-20&	3.99&	4.88E-20&	3.99&	2.75E-20&	4.00\\ 
				& 12&	4.69E-20&	3.99&	2.36E-20&	3.99&	1.33E-20&	3.96\\ 
				\midrule
				\multirow{5}{*}{$v$} 
				& 4&	1.17E-12&	--&	6.61E-13&	--&	5.28E-13&	--\\ 
				& 6&	2.28E-13&	4.02&	1.30E-13&	4.02&	1.05E-13&	3.99\\ 
				& 8&	7.16E-14&	4.03&	4.09E-14&	4.01&	3.35E-14&	3.97\\ 
				& 10&	2.90E-14&	4.05&	1.66E-14&	4.03&	1.41E-14&	3.86\\ 
				& 12&	1.39E-14&	4.04&	8.04E-15&	3.99&	7.03E-15&	3.84\\ 
				\midrule
				\multirow{5}{*}{$A$} 
				& 4&	7.11E-13&	--&	3.73E-13&	--&	2.57E-13&	--\\ 
				& 6&	1.40E-13&	4.01&	7.36E-14&	4.01&	5.08E-14&	4.00\\ 
				& 8&	4.39E-14&	4.03&	2.31E-14&	4.03&	1.59E-14&	4.03\\ 
				& 10&	1.76E-14&	4.11&	9.24E-15&	4.10&	6.53E-15&	4.00\\ 
				& 12&	8.03E-15&	4.30&	4.23E-15&	4.29&	3.16E-15&	3.98\\   
				\midrule
				\multirow{5}{*}{$B$} 
				& 4&	1.09E-09&	--&	5.74E-10&	--&	4.02E-10&	--\\ 
				& 6&	2.15E-10&	4.00&	1.13E-10&	4.00&	7.91E-11&	4.01\\ 
				& 8&	6.82E-11&	4.00&	3.59E-11&	4.00&	2.50E-11&	4.01\\ 
				& 10&	2.79E-11&	4.00&	1.47E-11&	4.00&	1.02E-11&	4.00\\ 
				& 12&	1.35E-11&	4.00&	7.09E-12&	4.00&	4.94E-12&	4.00\\   
				\bottomrule[1.5pt]
			\end{tabular}
		\end{table}
		
	\end{exmp}
	
	\begin{exmp}[TOV model]\label{ex:TOV}
		The TOV model describes static singular isothermal spheres and yields a time-independent metric of the form
		\begin{equation}\label{eq:TOVmetric}
			A(t,r)=1-8\pi\mathcal{G}\,\gamma,\qquad
			B(t,r)=B_0\,r^{\frac{4\tilde{\sigma}^2}{1+\tilde{\sigma}^2}}
			\quad\text{in \eqref{eq:metric}},
		\end{equation}
		where
		\begin{equation}\label{eq:TOVpara}
			\gamma=\frac{1}{2\pi\mathcal{G}}\left(\frac{\tilde{\sigma}^2}{1+6\tilde{\sigma}^2+\tilde{\sigma}^4}\right),\qquad
			\tilde{\sigma}=\frac{1}{\sqrt{3}},\qquad
			B_0=1.
		\end{equation}
		The exact solution is
		\begin{equation}\label{eq:solTOV}
			\rho(t,r)=\frac{\gamma}{r^2},\qquad v(t,r)=0,
		\end{equation}
		with boundary data prescribed as \((\rho(t,r_L),v(t,r_L))\) and \((\rho(t,r_R),v(t,r_R))\).
		
		The simulation runs over \(t\in[15,16]\) on \(\Omega=[3,7]\). The \(L^1\), \(L^2\), and \(L^\infty\) errors for the \(k\)th-order CPcOEDG schemes ($k = 2, 3, 4$) are reported in \cref{tb:4_2}–\cref{tb:4_4}. For comparison, \cref{tb:4_3_altOE} lists the density errors for the fourth-order OEDG scheme, which exhibits only first-order accuracy in all three norms due to the uniform damping coefficients in the OE procedure, as discussed in \cref{rem:OEDGandcOEDG}.
		
		\begin{table}[htp]\label{tb:4_2}
			\centering
			\belowrulesep=0pt
			\aboverulesep=0pt
			\caption{Errors and convergence rates of the second-order CPcOEDG method for \cref{ex:TOV}.}

			\setlength{\tabcolsep}{3mm}
			\begin{tabular}{c|cccccccc}
				\toprule[1.5pt]
				\multirow{2}{*}{} &
				\multirow{2}{*}{$N$} &
				\multicolumn{2}{c}{$L^1$ norm} &
				\multicolumn{2}{c}{$L^2$ norm} &
				\multicolumn{2}{c}{$L^\infty$ norm} \\
				\cmidrule(r){3-4} \cmidrule(r){5-6} \cmidrule(l){7-8}
				& & error & order &  error & order &  error & order \\    
				\midrule[1.5pt]
				\multirow{4}{*}{$\rho$} 
				& 100&	5.60E-08&	--&	5.29E-08&	--&	1.09E-07&	--\\ 
				& 200&	1.32E-08&	2.09&	1.27E-08&	2.06&	2.23E-08&	2.28\\ 
				& 400&	3.18E-09&	2.05&	3.17E-09&	2.00&	5.08E-09&	2.13\\ 
				& 800&	7.81E-10&	2.03&	7.94E-10&	2.00&	1.17E-09&	2.12\\ 
				\midrule
				\multirow{4}{*}{$v$} 
				& 100&	3.09E-05&	--&	2.37E-05&	--&	4.05E-05&	--\\ 
				& 200&	7.38E-06&	2.06&	5.43E-06&	2.12&	8.99E-06&	2.17\\ 
				& 400&	1.67E-06&	2.14&	1.24E-06&	2.13&	2.00E-06&	2.17\\ 
				& 800&	4.33E-07&	1.95&	3.13E-07&	1.98&	4.96E-07&	2.01\\ 
				\midrule
				\multirow{4}{*}{$A$} 
				& 100&	2.04E-06&	--&	1.62E-06&	--&	3.47E-06&	--\\ 
				& 200&	4.74E-07&	2.11&	4.61E-07&	1.82&	9.92E-07&	1.81\\ 
				& 400&	1.17E-07&	2.02&	1.25E-07&	1.88&	2.62E-07&	1.92\\ 
				& 800&	2.96E-08&	1.98&	3.27E-08&	1.94&	6.72E-08&	1.96\\    
				\midrule
				\multirow{4}{*}{$B$} 
				& 100&	4.42E-05&	--&	2.47E-05&	--&	2.14E-05&	--\\ 
				& 200&	9.52E-06&	2.22&	5.39E-06&	2.20&	4.81E-06&	2.15\\ 
				& 400&	2.19E-06&	2.12&	1.25E-06&	2.11&	1.14E-06&	2.08\\ 
				& 800&	5.22E-07&	2.07&	2.99E-07&	2.06&	2.76E-07&	2.04\\  
				\bottomrule[1.5pt]
			\end{tabular}
		\end{table}
		
		\begin{table}[htp]\label{tb:4_3}
			\centering
			\belowrulesep=0pt
			\aboverulesep=0pt
			\caption{Errors and convergence rates of the third-order CPcOEDG method for \cref{ex:TOV}.}
			
			\setlength{\tabcolsep}{3mm}
			\begin{tabular}{c|cccccccc}
				\toprule[1.5pt]
				\multirow{2}{*}{} &
				\multirow{2}{*}{$N$} &
				\multicolumn{2}{c}{$L^1$ norm} &
				\multicolumn{2}{c}{$L^2$ norm} &
				\multicolumn{2}{c}{$L^\infty$ norm} \\
				\cmidrule(r){3-4} \cmidrule(r){5-6} \cmidrule(l){7-8}
				& & error & order &  error & order &  error & order \\    
				\midrule[1.5pt]
				\multirow{4}{*}{$\rho$} 
				& 100&	3.25E-10&	--&	3.27E-10&	--&	6.47E-10&	--\\ 
				& 200&	4.05E-11&	3.00&	4.06E-11&	3.01&	8.11E-11&	3.00\\ 
				& 400&	5.01E-12&	3.02&	5.02E-12&	3.02&	1.01E-11&	3.00\\ 
				& 800&	6.22E-13&	3.01&	6.24E-13&	3.01&	1.26E-12&	3.01\\  
				\midrule
				\multirow{4}{*}{$v$} 
				& 100&	1.14E-07&	--&	1.02E-07&	--&	1.57E-07&	--\\ 
				& 200&	1.49E-08&	2.94&	1.30E-08&	2.97&	2.09E-08&	2.91\\ 
				& 400&	1.79E-09&	3.06&	1.59E-09&	3.02&	2.53E-09&	3.05\\ 
				& 800&	2.24E-10&	3.00&	1.99E-10&	3.00&	3.15E-10&	3.00\\
				\midrule
				\multirow{4}{*}{$A$} 
				& 100&	2.37E-08&	--&	2.18E-08&	--&	4.67E-08&	--\\ 
				& 200&	2.97E-09&	3.00&	2.73E-09&	2.99&	5.99E-09&	2.96\\ 
				& 400&	3.72E-10&	3.00&	3.43E-10&	3.00&	7.56E-10&	2.98\\ 
				& 800&	4.65E-11&	3.00&	4.29E-11&	3.00&	9.50E-11&	2.99\\     
				\midrule
				\multirow{4}{*}{$B$} 
				& 100&	1.53E-07&	--&	8.52E-08&	--&	7.38E-08&	--\\ 
				& 200&	2.00E-08&	2.94&	1.11E-08&	2.94&	9.54E-09&	2.95\\ 
				& 400&	2.54E-09&	2.97&	1.41E-09&	2.98&	1.21E-09&	2.98\\ 
				& 800&	3.21E-10&	2.99&	1.77E-10&	2.99&	1.52E-10&	2.99\\  
				\bottomrule[1.5pt]
			\end{tabular}
		\end{table}
		
		\begin{table}[htp]\label{tb:4_4}
			\centering
			\belowrulesep=0pt
			\aboverulesep=0pt
			\caption{Errors and convergence rates of the fourth-order CPcOEDG method for \cref{ex:TOV}.}
			
			\setlength{\tabcolsep}{3mm}
			\begin{tabular}{c|cccccccc}
				\toprule[1.5pt]
				\multirow{2}{*}{} &
				\multirow{2}{*}{$N$} &
				\multicolumn{2}{c}{$L^1$ norm} &
				\multicolumn{2}{c}{$L^2$ norm} &
				\multicolumn{2}{c}{$L^\infty$ norm} \\
				\cmidrule(r){3-4} \cmidrule(r){5-6} \cmidrule(l){7-8}
				& & error & order &  error & order &  error & order \\    
				\midrule[1.5pt]
				\multirow{4}{*}{$\rho$} 
				& 25&	4.49E-10&	--&	3.80E-10&	--&	5.59E-10&	--\\ 
				& 50&	2.38E-11&	4.24&	2.22E-11&	4.10&	3.46E-11&	4.02\\ 
				& 100&	1.36E-12&	4.13&	1.37E-12&	4.01&	2.20E-12&	3.97\\ 
				& 200&	8.14E-14&	4.06&	8.59E-14&	4.00&	1.39E-13&	3.99\\ 
				\midrule
				\multirow{4}{*}{$v$} 
				& 25&	2.50E-07&	--&	2.45E-07&	--&	5.14E-07&	--\\ 
				& 50&	1.12E-08&	4.48&	1.09E-08&	4.50&	2.33E-08&	4.46\\ 
				& 100&	5.79E-10&	4.28&	5.56E-10&	4.29&	1.15E-09&	4.35\\ 
				& 200&	3.31E-11&	4.13&	3.14E-11&	4.15&	6.19E-11&	4.21\\  
				\midrule
				\multirow{4}{*}{$A$} 
				& 25&	3.31E-08&	--&	3.94E-08&	--&	1.08E-07&	--\\ 
				& 50&	1.59E-09&	4.38&	1.77E-09&	4.48&	5.31E-09&	4.35\\ 
				& 100&	8.71E-11&	4.19&	9.21E-11&	4.26&	2.76E-10&	4.26\\ 
				& 200&	5.10E-12&	4.10&	5.25E-12&	4.13&	1.54E-11&	4.17\\   
				\midrule
				\multirow{4}{*}{$B$} 
				& 25&	3.29E-07&	--&	1.76E-07&	--&	1.25E-07&	--\\ 
				& 50&	8.08E-09&	5.35&	4.28E-09&	5.36&	2.93E-09&	5.42\\ 
				& 100&	1.35E-10&	5.90&	7.27E-11&	5.88&	5.89E-11&	5.64\\ 
				& 200&	5.06E-12&	4.74&	2.69E-12&	4.76&	2.46E-12&	4.58\\   
				\bottomrule[1.5pt]
			\end{tabular}
		\end{table}
		
		\begin{table}[htp]\label{tb:4_3_altOE}
			\centering
			\belowrulesep=0pt
			\aboverulesep=0pt
			\caption{Errors and convergence rates for the density $\rho$ of the {fourth}-order CPcOEDG method with the original OE coefficients of \cite{peng2025oscillation}.}
			\setlength{\tabcolsep}{3mm}
			\begin{tabular}{c|cccccccc}
				\toprule[1.5pt]
				\multirow{2}{*}{} &
				\multirow{2}{*}{$N$} &
				\multicolumn{2}{c}{$L^1$ norm} &
				\multicolumn{2}{c}{$L^2$ norm} &
				\multicolumn{2}{c}{$L^\infty$ norm} \\
				\cmidrule(r){3-4} \cmidrule(r){5-6} \cmidrule(l){7-8}
				& & error & order &  error & order &  error & order \\    
				\midrule[1.5pt]
				\multirow{3}{*}{$\rho$} 
				& 25&	5.12E-05&	--&	3.44E-05&	--&	5.92E-05&	--\\ 
				& 50&	2.62E-05&	0.97&	1.77E-05&	0.96&	3.08E-05&	0.94\\ 
				& 100&	1.31E-05&	1.00&	8.98E-06&	0.98&	1.61E-05&	0.94\\  
				\bottomrule[1.5pt]
			\end{tabular}
		\end{table}
	\end{exmp}
	
	\subsection{Physical fidelity tests}
	Beyond accuracy on smooth solutions, we demonstrate that the proposed method is {CP} and effectively eliminates spurious oscillations through several benchmark problems, including accretion onto a Schwarzschild black hole (\cref{ex:Schwarzschild}), a shock wave model (\cref{ex:shock}), and a time-reversal test (\cref{ex:reverse}).
	
	\begin{exmp}[Accretion onto a Schwarzschild black hole]\label{ex:Schwarzschild}
		We perform a long-time simulation up to \(t=160\) using the fourth-order CPcOEDG scheme to study the steady state for a Schwarzschild black hole of mass $M = 1$ with right-boundary inflow. The strong gravitational field of the Schwarzschild black hole curves the spacetime metric to
		\[
		A(t,r)=B(t,r)=1-\frac{2}{r},\qquad  \Omega = [2.2, 20.2],
		\]
		and the Einstein coupling constant is set to \(\kappa=0\).
		The analytic steady state is
		\begin{equation}\label{eq:steadyflow}
			\rho(r) = \frac{D_0(1-v^2)}{-vr(r-2)}, \quad v(r) = -\sqrt{\varpi(r)}, \quad D_0 = 0.016,
		\end{equation}
		where \(\varpi(r)\) satisfies 
		\begin{equation}\label{eq:st}
			(1-\varpi)\,\varpi^{\frac{\tilde{\sigma}^2}{1-\tilde{\sigma}^2}}
			= \left(1-\frac{2}{r}\right)\!
			\left(\frac{2}{r}\right)^{\frac{4\tilde{\sigma}^2}{1-\tilde{\sigma}^2}},\quad \tilde{\sigma} = 0.1 .
		\end{equation}
		At the right boundary we prescribe \((\rho(r_R),v(r_R))\) from \eqref{eq:steadyflow} with \(r=r_R\), and impose a pure outflow condition at the left boundary. The domain \(\Omega\) is discretized into \(400\) uniform cells. 
		
		Figure~\ref{fig:0} shows that the scheme remains stable over the entire simulation, preserves admissibility even as $v$ approaches the speed of light near the black hole, and accurately captures the steady state.
		
		\begin{figure}[htp]
			\centering
			\label{fig:0}\includegraphics[width=0.48\linewidth]{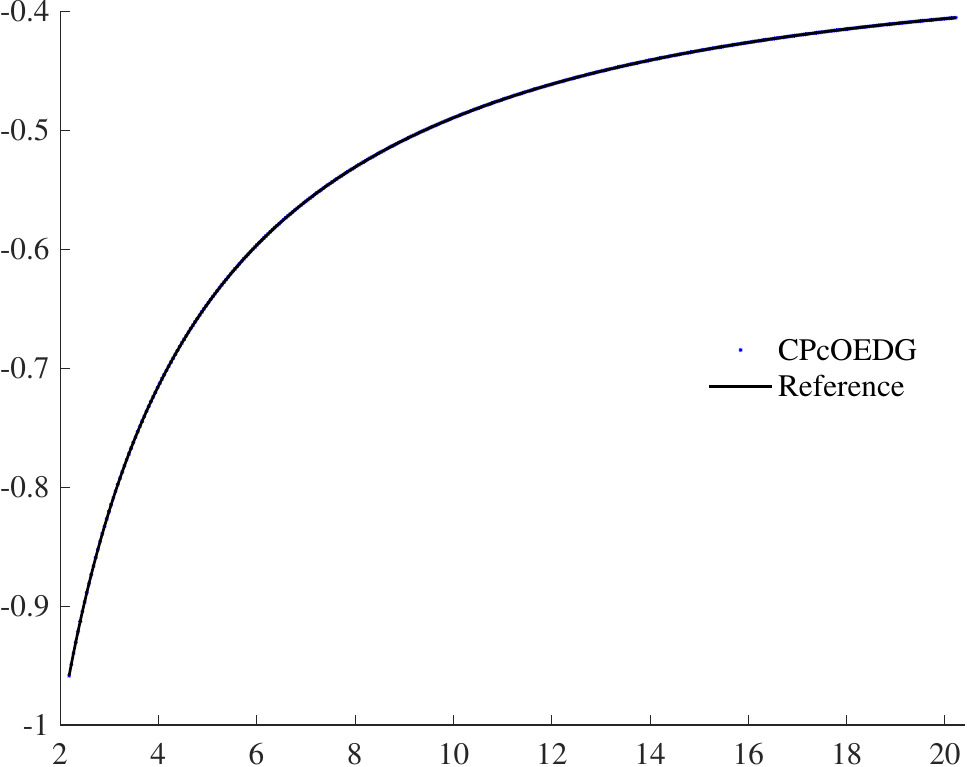}
			\includegraphics[width=0.48\linewidth]{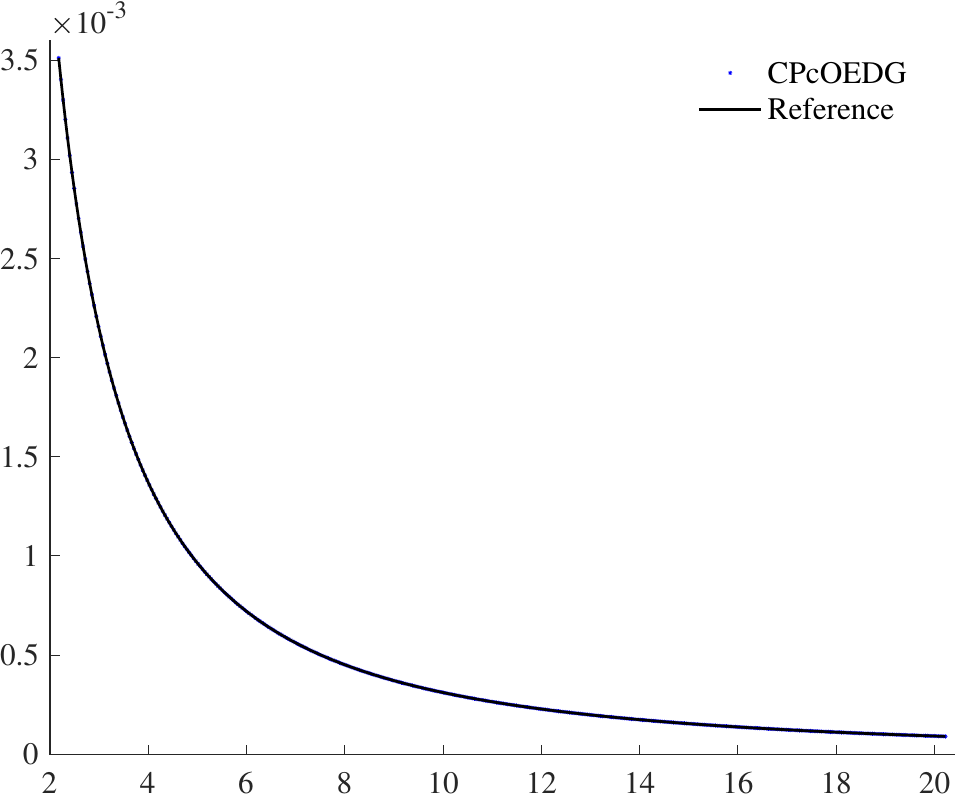}
			\caption{Velocity $v$ (left) and density $\rho$ (right) at $t = 160$ computed by the fourth-order CPcOEDG method for \cref{ex:Schwarzschild}.}
		\end{figure}

	\end{exmp}

	\subsection{Example 4.5: Shock wave model}

	\begin{exmp}[Shock wave model \cite{VoglerTemple2012}]\label{ex:shock}
		We next consider a Riemann-type problem on the domain $\Omega = [3, 7]$. The simulation is initialized at $t = t_0$ and evolved to $t = t_0 + 1$. The initial primitive variables exhibit a discontinuity at $r_0 = 5$. The spacetime is constructed by matching an FRW-1 metric (left) with a TOV metric (right) across the interface $r_0$, in such a way that the metric components are continuous. The initial conditions for the fluid are
		\begin{equation}
			\bigl(\rho(t_0,r),\, v(t_0,r)\bigr)=
			\begin{cases}
				\left(\dfrac{3 v^2}{\kappa r^2},\, \dfrac{1-\sqrt{1-\xi^2}}{\xi}\right), & r<r_0,\\[6pt]
				\left(\dfrac{\gamma}{r^2},\, 0\right), & r>r_0,
			\end{cases}
		\end{equation}
		with $\xi=r/t$ and the same parameter $\gamma$ as in \cref{ex:TOV}.
		
		The piecewise metric coefficients at $t=t_0$ are
		\begin{equation}
			\begin{aligned}
				A(t_0,r) &=
				\begin{cases}
					1-v^2, & r<r_0,\\
					1-8\pi\mathcal{G}\,\gamma, & r>r_0,
				\end{cases}
				\qquad
				B(t_0,r) &=
				\begin{cases}
					\dfrac{1}{1-v^2}, & r<r_0,\\[6pt]
					B_0\, r^{\frac{4\tilde{\sigma}^2}{1+\tilde{\sigma}^2}}, & r>r_0,
				\end{cases}
			\end{aligned}
		\end{equation}
		where $\tilde{\sigma}=1/\sqrt{3}$. Enforcing continuity of the metric at $r_0=5$ fixes
		\[
		v_0:=v(t_0,r_0^-)=\sqrt{8\pi\mathcal{G}\,\gamma},
		\qquad
		t_0=\frac{r_0\bigl(1+v_0^2\bigr)}{2v_0},
		\qquad
		B_0=\frac{r_0^{-\frac{4\tilde{\sigma}^2}{\,1+\tilde{\sigma}^2\,}}}{1-v_0^2}.
		\]
		
		On the spatial boundaries $r_R = 7$ and $r_L = 3$, the primitive variables are set as
		\[
		(\rho_L(t),v_L(t))=\left(\frac{3\,v_L(t)^2}{\kappa r_L^2},\ \frac{1-\sqrt{1-(r_L/t)^2}}{r_L/t}\right), 
		\quad
		(\rho_R(t),v_R(t))=\left(\frac{\gamma}{r_R^2},\ 0\right).
		\]
		
		This setup generates two shock waves: a stronger shock propagating into the TOV region and a weaker shock moving into the FRW-1 region.  
		We compare the density $\rho$ and velocity $v$ computed by the fourth-order CPcOEDG method on $400$ uniform cells with those from Godunov schemes on $10{,}000$ uniform cells at $t=t_0+0.2$, $t_0+0.6$, and $t_0+1$ in \cref{fig:7.2,fig:7.6,fig:7.10}. The scheme maintains CP, sharply resolves the discontinuities, and exhibits no spurious oscillations in the primitive variables---in contrast to conventional CPcRKDG schemes. 
		
		In addition, \cref{fig:7.10AB} shows the metric functions $A_h$ and $B_h$ at $t=t_0+1$, confirming preservation of metric continuity. Overall, both primitive and metric variables remain free of spurious oscillations, demonstrating the {OE} capability of the proposed method for shock-dominated problems.
		
		\begin{figure}[htp]
			\centering
			\label{fig:7.2}\includegraphics[width=0.48\linewidth]{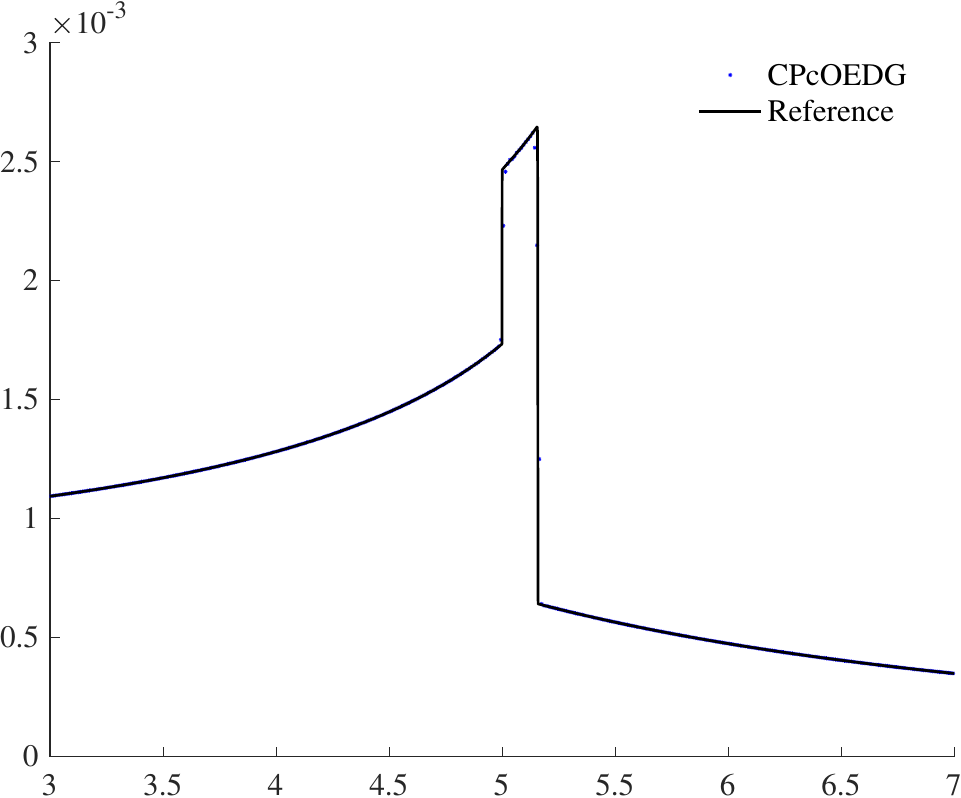}
			\includegraphics[width=0.48\linewidth]{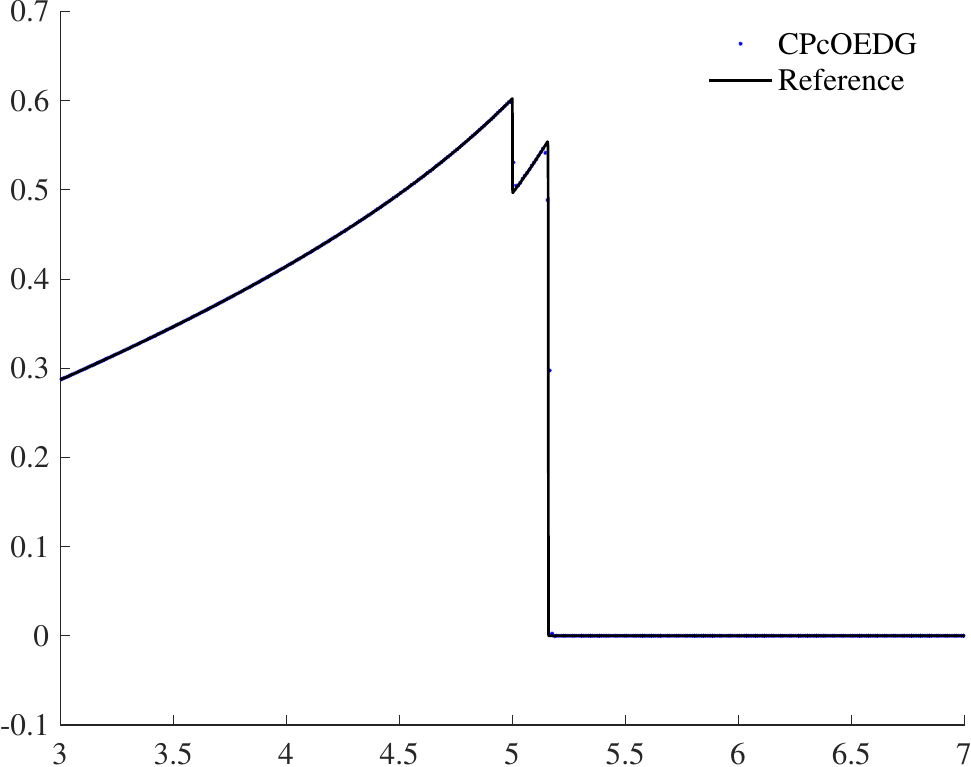}
			\caption{
				Density (left) and velocity (right) for \cref{ex:shock} computed by the fourth-order CPcOEDG method at $t = t_0+0.2$. 
			}
		\end{figure}
		
		\begin{figure}[htp]
			\centering
			\label{fig:7.6}\includegraphics[width=0.48\linewidth]{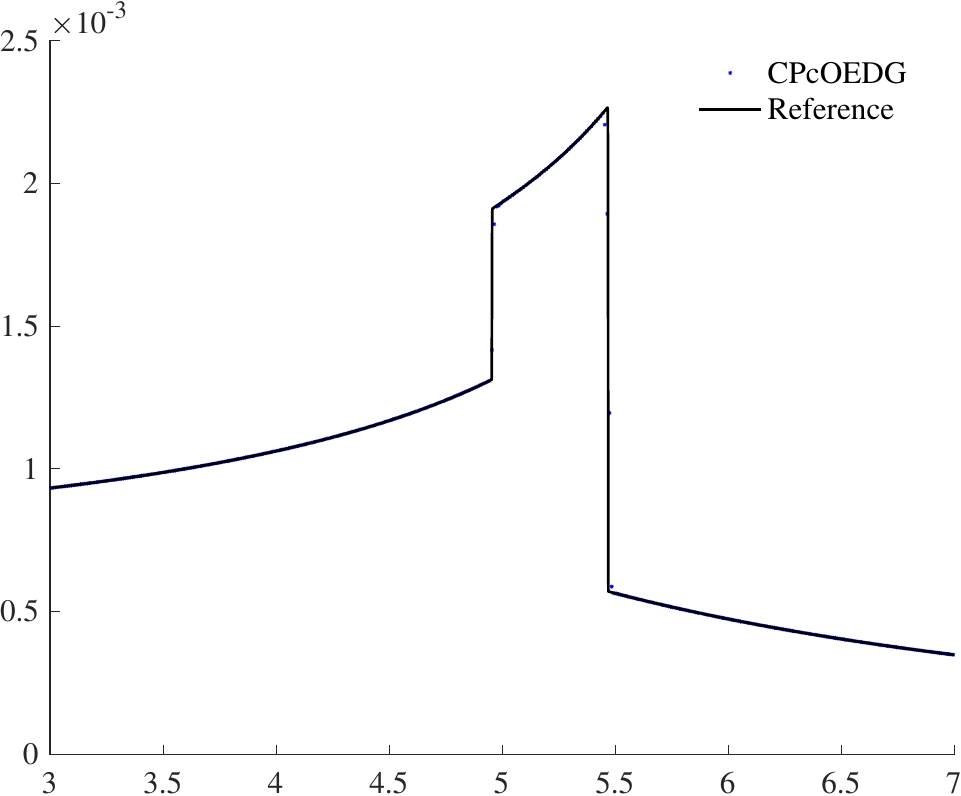}
			\includegraphics[width=0.48\linewidth]{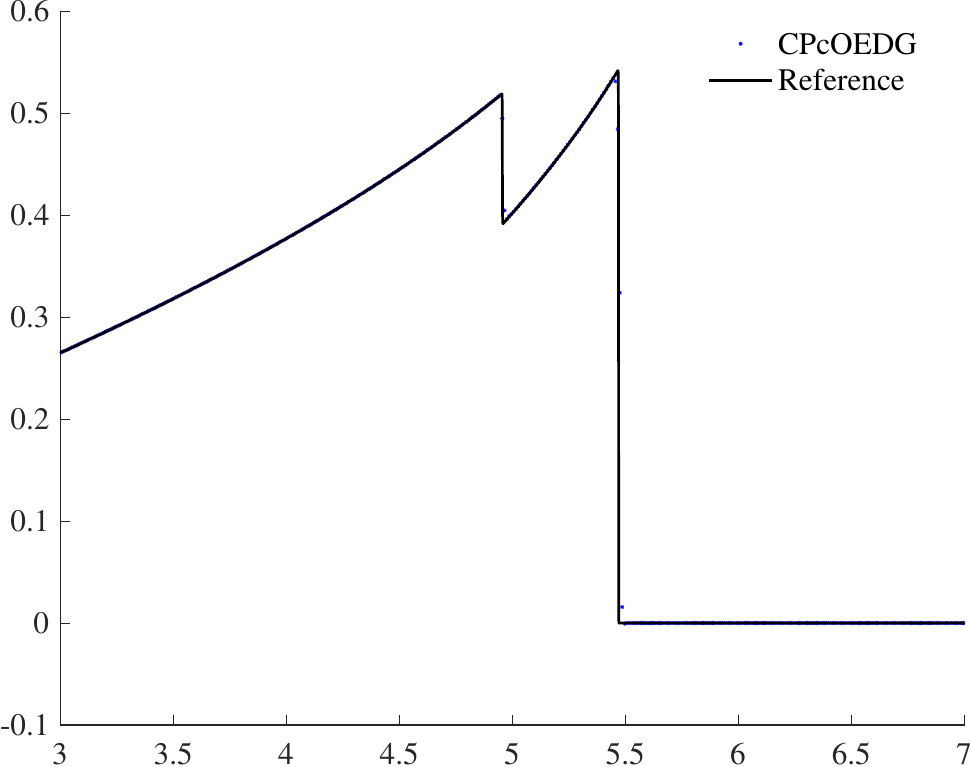}
			\caption{
				Density (left) and velocity (right) for \cref{ex:shock} computed by the fourth-order CPcOEDG method at $t = t_0 + 0.6$.}
		\end{figure}
		
		\begin{figure}[htp]
			\centering
			\label{fig:7.10}\includegraphics[width=0.48\linewidth]{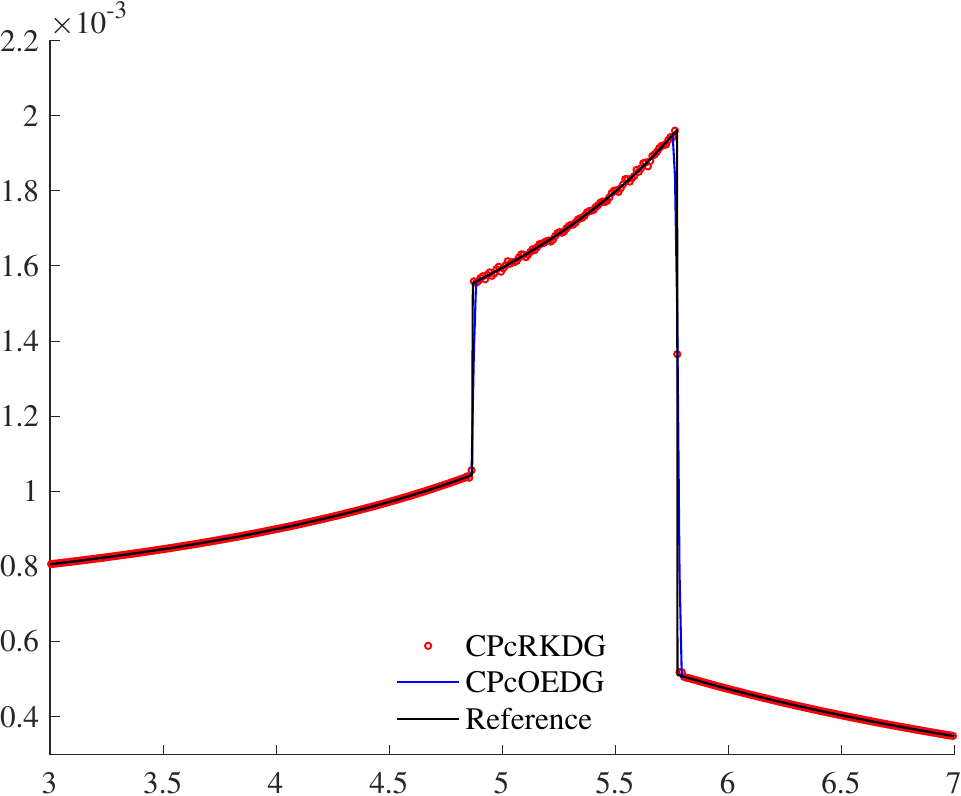}
			\includegraphics[width=0.48\linewidth]{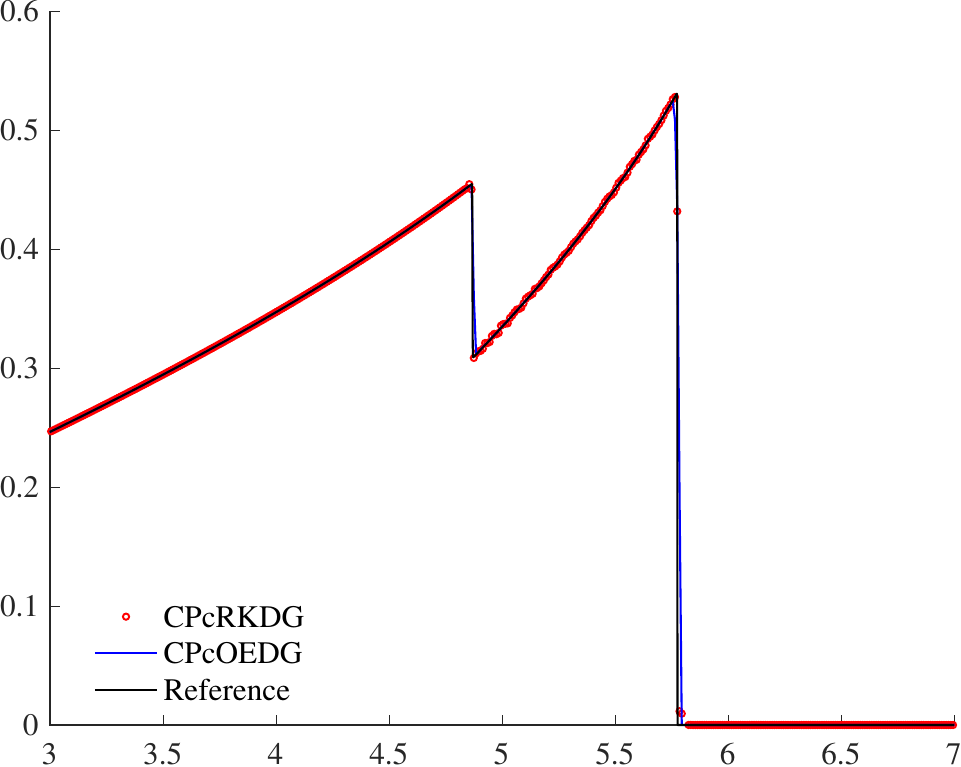}
			\caption{Velocity $v$ (left) and density $\rho$ (right) at $t = t_0 + 1$ obtained with the fourth-order CPcRKDG and CPcOEDG methods for \cref{ex:shock}.} 
		\end{figure}
		
		\begin{figure}[htp]
			\centering
			\label{fig:7.10AB}\includegraphics[width=0.48\linewidth]{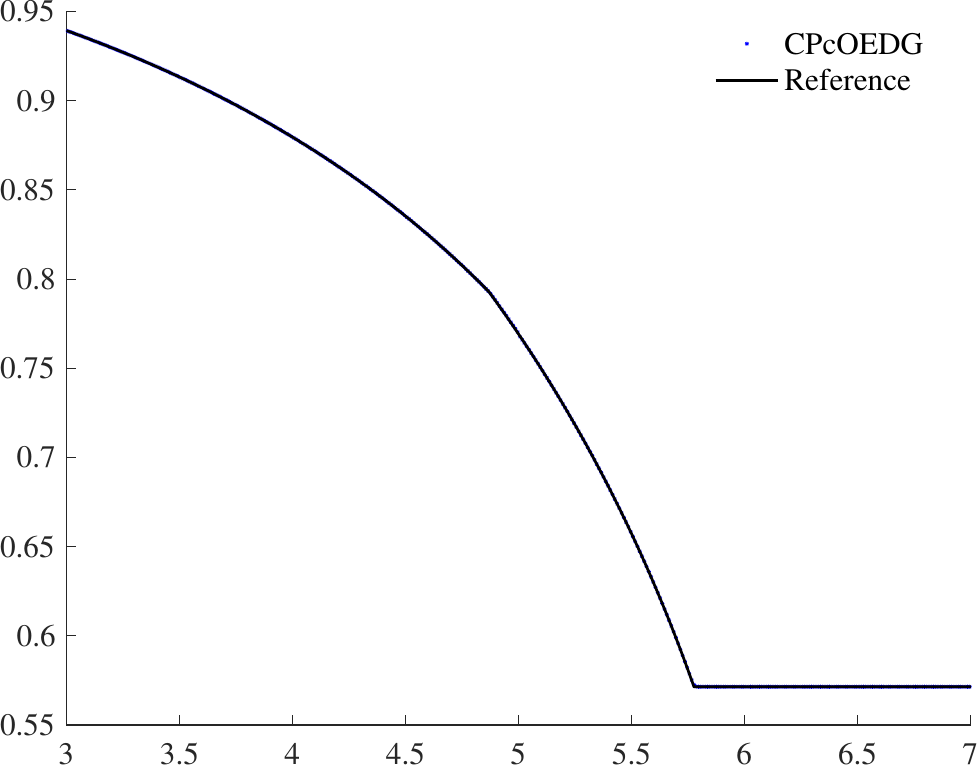}
			\includegraphics[width=0.48\linewidth]{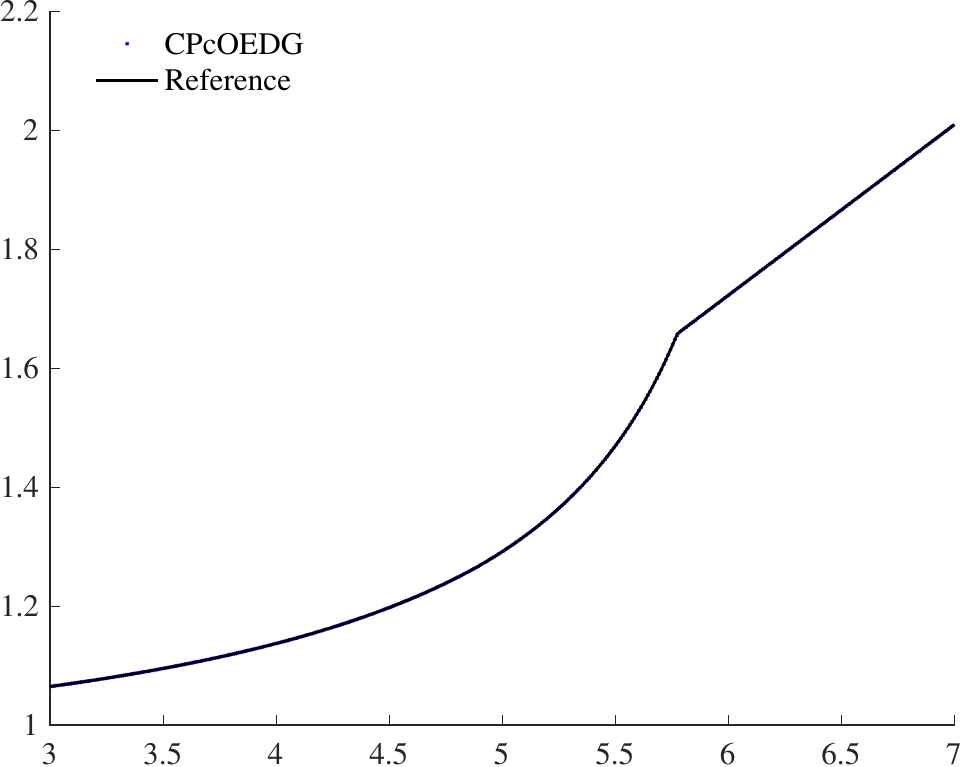}
			\caption{
				Metric variables $A$ (left) and $B$ (right) at $t = t_0 + 1$ for \cref{ex:shock}, computed by the fourth-order CPcOEDG method.}
		\end{figure}
	\end{exmp}

	\begin{exmp}[Time reversal model]\label{ex:reverse}
		Finally, we repeat \cref{ex:shock} with the time direction reversed. The initial conditions are identical to those in \cref{ex:shock}, except that $t_0 = -\frac{r_0(1+v_0^2)}{2v_0}$, which results in the formation of two rarefaction waves at early times.  
		
		Figures~\ref{fig:8.2}, \ref{fig:8.6}, and \ref{fig:8.10} show the numerical solutions of the conservative variables $\mathcal{T}^{00}$ and $\mathcal{T}^{01}$ obtained by the fourth-order CPcOEDG scheme at $t = t_0+0.2$, $t = t_0+0.6$, and $t = t_0+1$, respectively, together with reference solutions from a Godunov scheme on $10{,}000$ uniform cells. 
		
		Figure~\ref{fig:8.10AB} presents the metric functions $A$ and $B$ computed by the CPcOEDG scheme at $t = t_0 + 1$. All numerical results agree closely with the reference solutions, maintain CP, and display no spurious oscillations, thereby further validating the effectiveness of our proposed method.
		
		\begin{figure}[htp]
			\centering
			\label{fig:8.2}\includegraphics[width=0.48\linewidth]{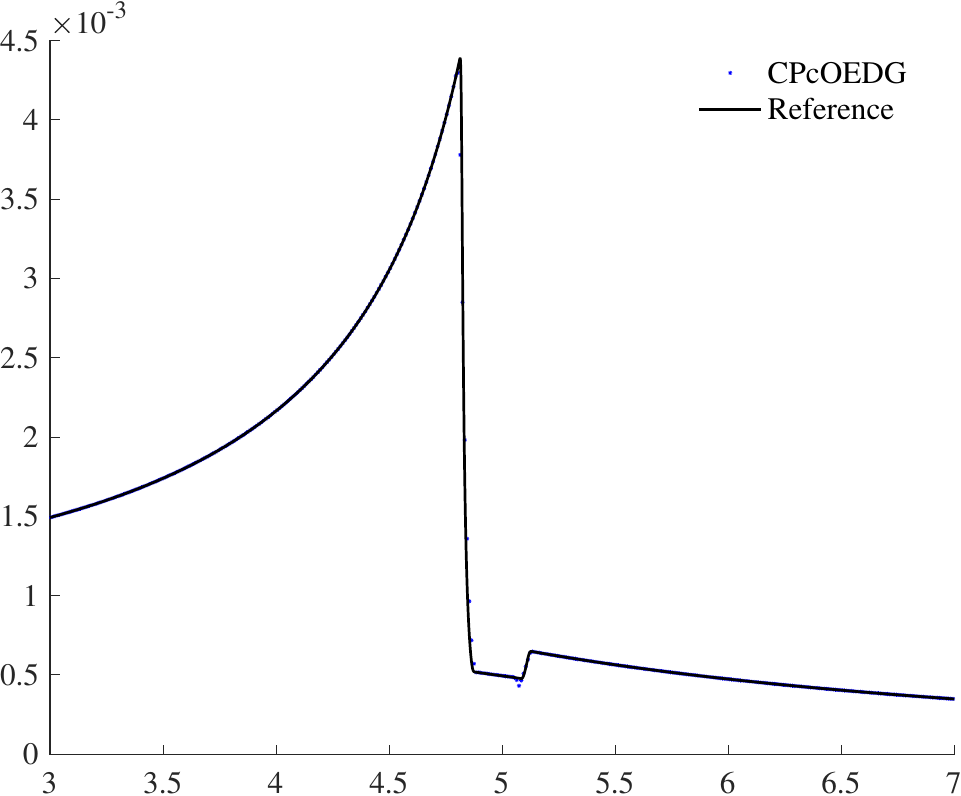}
			\includegraphics[width=0.48\linewidth]{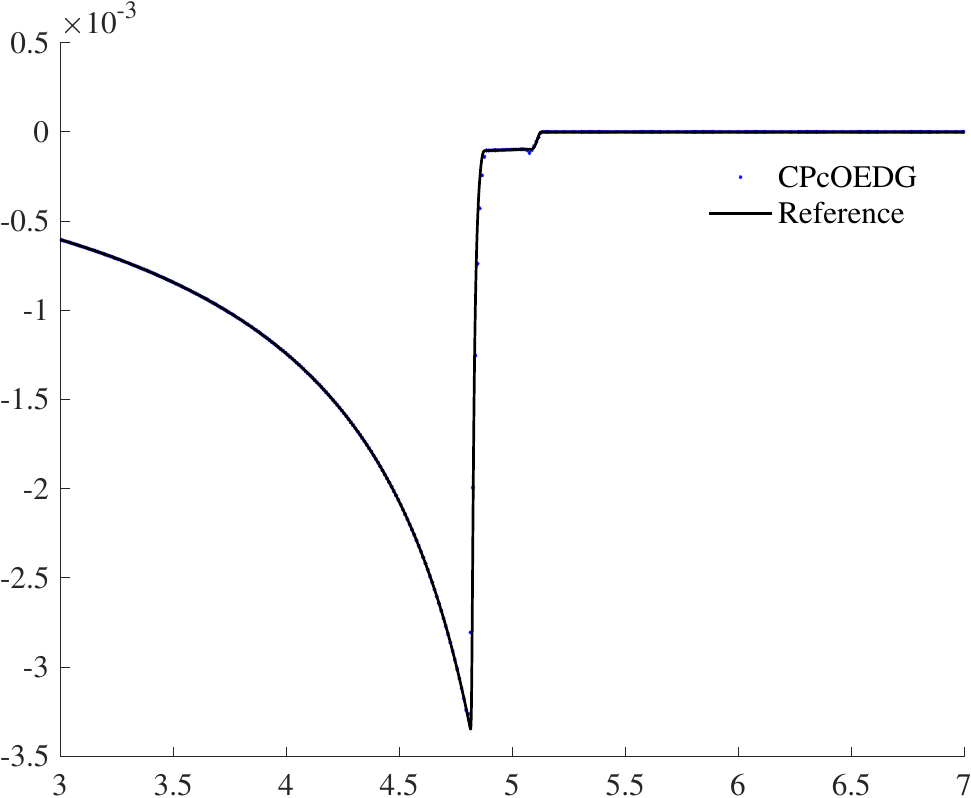}
			\caption{
				Numerical solution of $\mathcal{T}^{00}$ (left) and $\mathcal{T}^{01}$ (right) at $t = t_0+0.2$ obtained with the fourth-order CPcOEDG method for \cref{ex:reverse}.}  
		\end{figure}
		
		\begin{figure}[htp]
			\centering
			\label{fig:8.6}\includegraphics[width=0.48\linewidth]{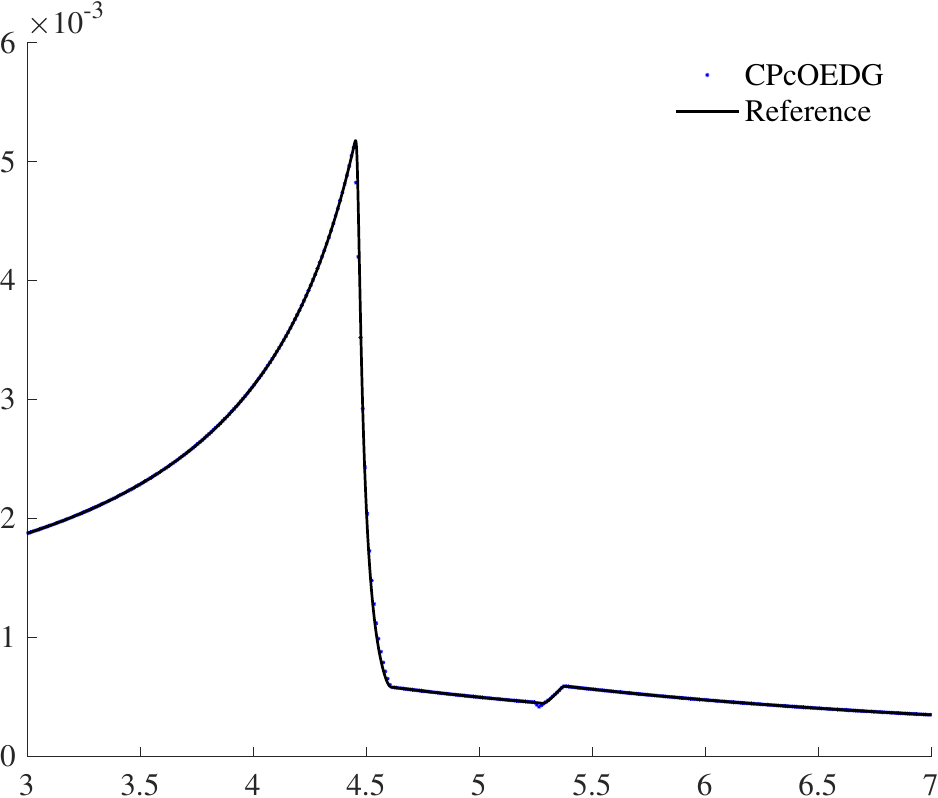}
			\includegraphics[width=0.48\linewidth]{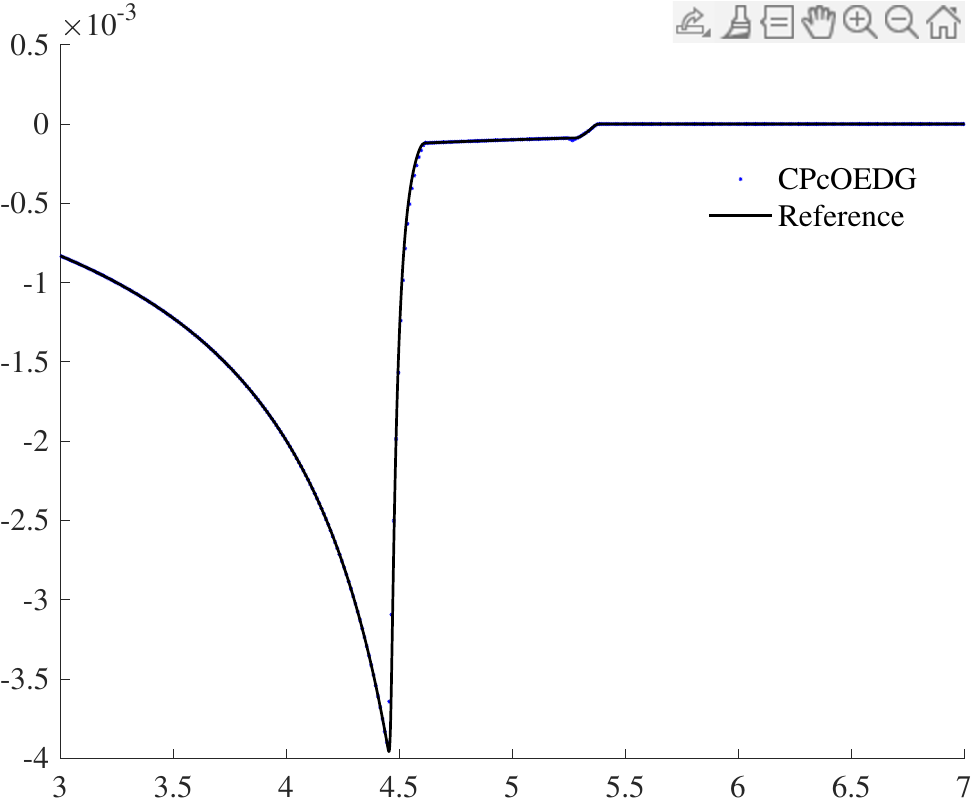}
			\caption{
				Conservative variables $\mathcal{T}^{00}$ (left) and $\mathcal{T}^{01}$ (right) at $t = t_0+0.6$ for \cref{ex:reverse}, computed with the fourth-order CPcOEDG method.}
		\end{figure}
		
		\begin{figure}[htp]
			\centering
			\label{fig:8.10}\includegraphics[width=0.48\linewidth]{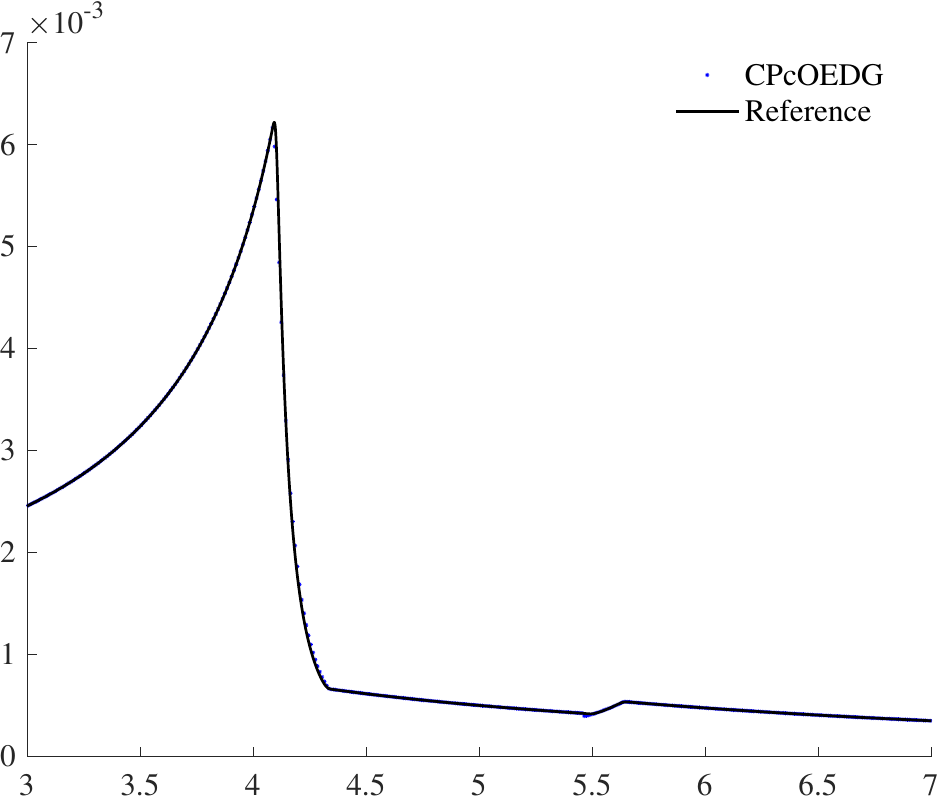}
			\includegraphics[width=0.48\linewidth]{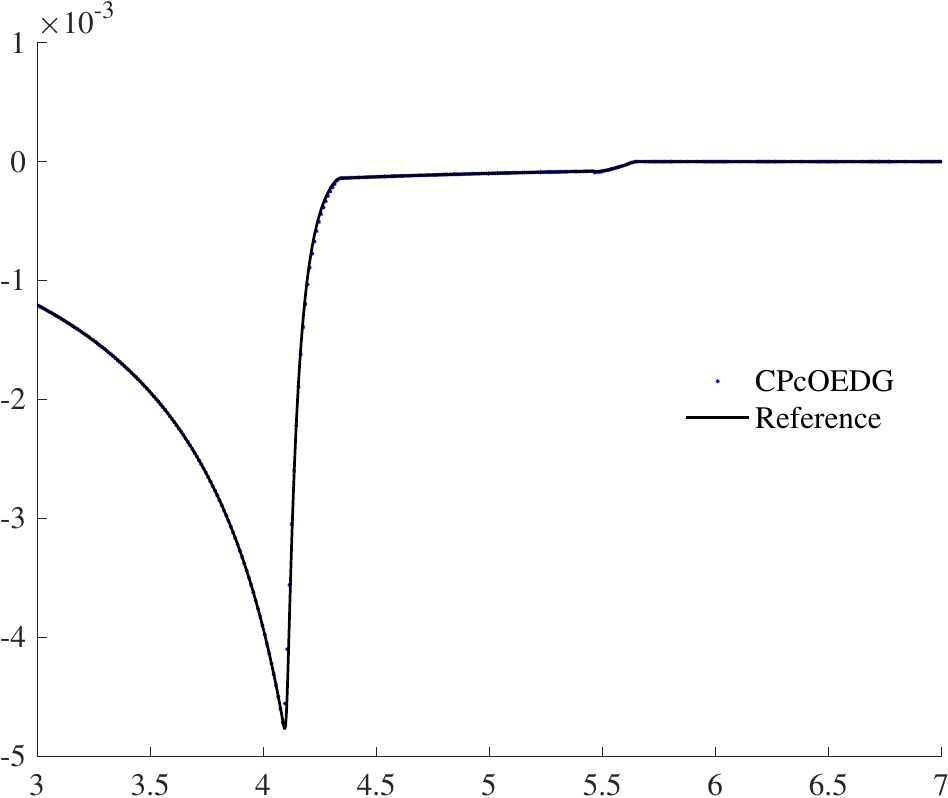}
			\caption{Conservative variables $\mathcal{T}^{00}$ (left) and $\mathcal{T}^{01}$ (right) at $t = t_0+1$ for \cref{ex:reverse}, computed with the fourth-order CPcOEDG method.}
		\end{figure}
		
		\begin{figure}[htp]
			\centering
			\label{fig:8.10AB}\includegraphics[width=0.48\linewidth]{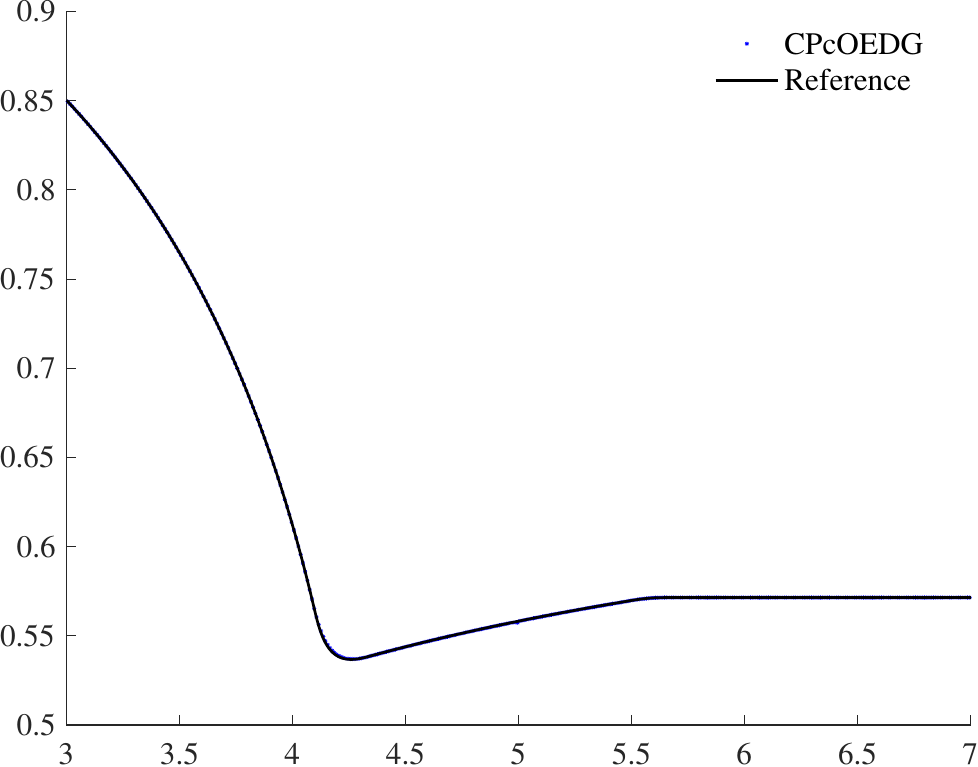}
			\includegraphics[width=0.48\linewidth]{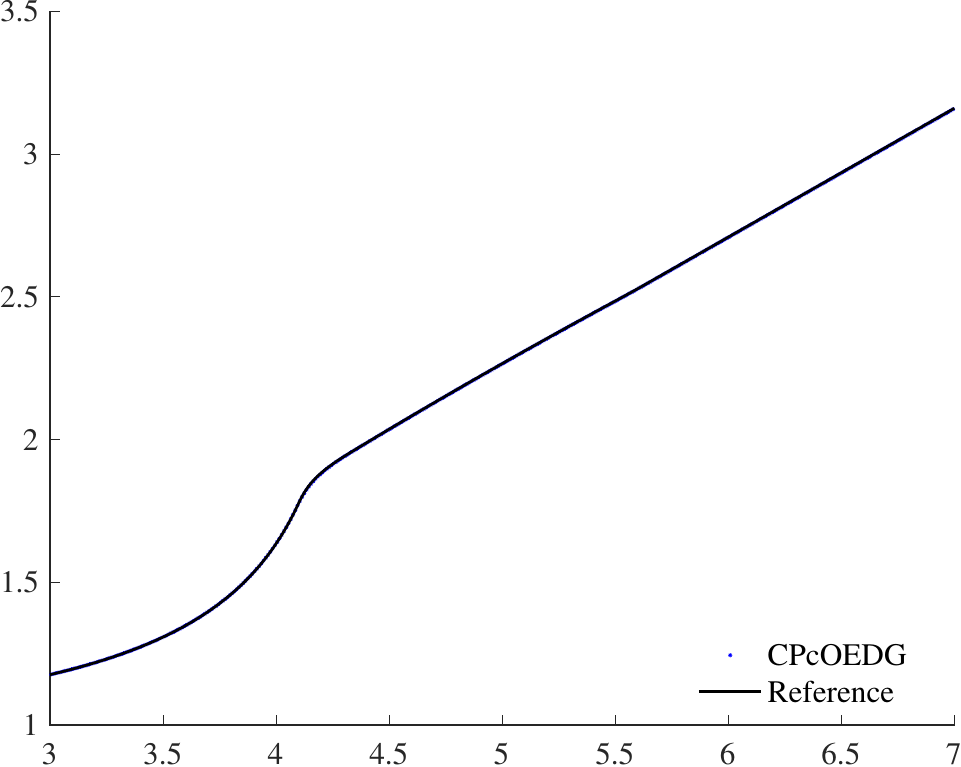}
			\caption{
				Metric functions $A$ (left) and $B$ (right) at $t = t_0 + 1$ for \cref{ex:reverse}, computed with the fourth-order CPcOEDG method.}
		\end{figure}
	\end{exmp}

	\section{Conclusions}
	\label{sec:conclusions}
	This paper has proposed a high-order constraint-preserving compact Oscillation-Eliminating Discontinuous Galerkin (CPcOEDG) method for the spherically symmetric EE system. The proposed scheme successfully addresses the dual challenges of maintaining relativistic hydrodynamic constraints and ensuring the geometric validity of the evolving spacetime metric.
	
	The core theoretical contribution is a rigorous characterization of the admissible state set. We proved that the physical constraints of relativistic fluids, specifically positivity of density and subluminal velocity, define a convex set directly in terms of the conservative variables. This observation enables a robust bound-preserving framework that operates without the complexity of primitive-variable checks. In addition, we introduced a bijective change of variables for the metric potentials, which intrinsically enforces positivity and asymptotic bounds on the metric. By coupling this geometric transformation with a hybrid update strategy and a compatible boundary treatment, our method avoids the order reduction commonly observed in standard RKDG schemes near boundaries.
	
	Extensive numerical experiments, including FRW cosmologies, TOV stars, and black hole accretion, have validated the performance of the scheme. The results confirm that the CPcOEDG method achieves design-order accuracy on smooth solutions while exhibiting strong robustness and oscillation-free shock-capturing capabilities in regimes with strong gravity--fluid interactions.
	
	Future work will focus on extending this constraint-preserving framework to multi-dimensional general relativistic hydrodynamics and exploring its application to more complex astrophysical scenarios involving fully dynamic spacetime evolution.
	
	\appendix
	\section{Numerical flux}\label{app:nf}
	The characteristic eigenvalues of the Jacobian matrix $\partial \left(\sqrt{A B} {\boldsymbol{F}}\right)/\partial \boldsymbol{U}$ are analytically given by
	\begin{displaymath}
		\lambda_1 = \sqrt{A B}\frac{v - \sqrt{P'(\rho)}}{1 - \sqrt{P'(\rho)} v}, \quad \lambda_2 = \sqrt{A B}\frac{v + \sqrt{P'(\rho)}}{1 + \sqrt{P'(\rho)} v}.
	\end{displaymath}
	
	\section{Proof of the violation of the LF splitting property}
	\begin{lemma}[Violation of LF splitting]\label{lem:LFviolation}
		Let $s_1(\boldsymbol U)<0<s_2(\boldsymbol U)$ be the eigenvalues of the Jacobian 
		$\partial \boldsymbol{F}(\boldsymbol{U}) / \partial \boldsymbol{U}$ defined in \eqref{equ:eigenF}. 
		Then, in general, the LF splitting based on $s_1,s_2$ does not preserve the admissible set $G_c$. More precisely, there exist states $\boldsymbol U\in G_c$ such that
		\[
		\boldsymbol{U}+\frac{1}{s_2(\boldsymbol U)}\,\boldsymbol{F}(\boldsymbol{U})\notin G_c,
		\qquad
		\boldsymbol{U}-\frac{1}{s_1(\boldsymbol U)}\,\boldsymbol{F}(\boldsymbol{U})\notin G_c.
		\]
	\end{lemma}
	
	\begin{proof}
		The sign condition $s_1(\boldsymbol U)<0<s_2(\boldsymbol U)$ implies, through the eigenvalue representation \eqref{equ:eigenF}, that the fluid velocity satisfies $v\in(-c_s,c_s)$.
		We first show that 
		\[
		\boldsymbol{U}+\frac{1}{s_2(\boldsymbol U)}\,\boldsymbol{F}(\boldsymbol{U})\notin G_c
		\]
		for a suitable admissible state $\boldsymbol U\in G_c$. 
		The argument for the other state
		$\boldsymbol U - s_1(\boldsymbol U)^{-1}\boldsymbol F(\boldsymbol U)$ is analogous and can be obtained by replacing $\boldsymbol F$ with $-\boldsymbol F$.
		
		Define
		\[
		\boldsymbol R^+ := \begin{pmatrix} R_1^+ \\ R_2^+ \end{pmatrix} 
		:= \boldsymbol U + \frac{1}{s_2(\boldsymbol U)}\,\boldsymbol F(\boldsymbol U).
		\]
		A direct computation shows that the scalar quantity
		\begin{equation}\label{eq:rootsofR}
			P_+(v) := R_1^+ + R_2^+ 
			= (1+v)\,Q_+(v),
		\end{equation}
		where
		\[
		Q_+(v) := (c_s\rho+p)(1+v^2) + 2v(\rho+c_s p),
		\]
		is a cubic polynomial in $v$. The factorization \eqref{eq:rootsofR} implies that $P_+(v)$ has three real roots $v_0,v_1,v_2$ satisfying
		\[
		v_0 v_2 = 1,\qquad
		v_0 + v_2 < 0,\qquad
		v_1 = -1.
		\]
		Hence $v_0< -1$ and $v_2\in(-1,0)$, while $v_1=-1$ is the remaining root.
		
		The admissible set $G_c$ is characterized (in this transformed frame) by positivity conditions of the form
		$R_1^+ \pm R_2^+>0$. Therefore, if we can find a velocity $v\in(-c_s,c_s)$ such that
		\[
		P_+(v) = R_1^+ + R_2^+ < 0,
		\]
		then the corresponding LF-split state $\boldsymbol R^+$ violates at least one of the defining inequalities of $G_c$.
		
		Since $P_+(v)=(1+v)Q_+(v)$ and $1+v>0$ for all $v\in(-1,1)$, it suffices to study the sign of $Q_+(v)$.
		Under the barotropic equation of state and the structural relation \eqref{eq:key_re}, we have
		\[
		\begin{aligned}
			Q_+(-c_s)
			&= (c_s\rho+p)(1+c_s^2)+2(-c_s)(\rho+c_s p)\\
			&= (c_s\rho+p) + c_s^2(c_s\rho+p) -2c_s\rho -2c_s^2 p\\
			&= (p-c_s\rho) + c_s^2(c_s\rho-p)\\
			&= (1-c_s^2)(p-c_s\rho).
		\end{aligned}
		\]
		Since $0<c_s<1$ and $c_s\rho>p$ by assumption, we have $1-c_s^2>0$ and $p-c_s\rho<0$, hence
		\[
		Q_+(-c_s) = (1-c_s^2)(p-c_s\rho) < 0.
		\]
		Because $-c_s\in(-1,0)$ and $1-c_s>0$, this implies
		\[
		P_+(-c_s) = (1-c_s)\,Q_+(-c_s) < 0.
		\]
		For such a state, if the velocity  $v$ approaches $v=-c_s$ from the interior of $(-c_s,c_s)$, we have $R_1^+ + R_2^+<0$, showing $\boldsymbol R^+\notin G_c$. The proof for
		$\boldsymbol U - s_1(\boldsymbol U)^{-1}\boldsymbol F(\boldsymbol U)$ follows by the same argument, after replacing $\boldsymbol F$ by $-\boldsymbol F$, and is omitted.
	\end{proof}
	
	\section{The fourth-order CPcOEDG method}\label{app:RK4}
	For completeness, we list the fourth-order CPcOEDG scheme following \cite{LiuSunZhang2025}:
	\begin{align*}
		\boldsymbol{U}_\sigma^{n} &= \mathcal{B}\,\mathcal{F}_{\tau}\,\boldsymbol{U}_h^n,\\[2pt]
		\int_{I_j} \boldsymbol{U}_h^{n,1}\!\cdot\!\boldsymbol{\phi}\,{\mathrm{d}r}
		&= \int_{I_j} \boldsymbol{U}_\sigma^{n}\!\cdot\!\boldsymbol{\phi}\,{\mathrm{d}r}
		+ \frac{\Delta t}{2}\,
		\mathcal{L}_j^{\mathrm{loc}}\!\bigl(\boldsymbol{U}_\sigma^{n,0},A_h^{n,0},B_h^{n,0};\boldsymbol{\phi}\bigr),\\[2pt]
		\int_{I_j} \boldsymbol{U}_h^{n,2}\!\cdot\!\boldsymbol{\phi}\,{\mathrm{d}r}
		&= \int_{I_j} \boldsymbol{U}_\sigma^{n}\!\cdot\!\boldsymbol{\phi}\,{\mathrm{d}r}
		+ \frac{\Delta t}{2}\,
		\mathcal{L}_j^{\mathrm{loc}}\!\bigl(\mathcal{B}\boldsymbol{U}_h^{n,1},A_h^{n,1},B_h^{n,1};\boldsymbol{\phi}\bigr),\\[2pt]
		\int_{I_j} \boldsymbol{U}_h^{n,3}\!\cdot\!\boldsymbol{\phi}\,{\mathrm{d}r}
		&= \int_{I_j} \boldsymbol{U}_\sigma^{n}\!\cdot\!\boldsymbol{\phi}\,{\mathrm{d}r}
		+ {\Delta t}\,
		\mathcal{L}_j^{\mathrm{loc}}\!\bigl(\mathcal{B}\boldsymbol{U}_h^{n,2},A_h^{n,2},B_h^{n,2};\boldsymbol{\phi}\bigr),\\[2pt]
		\int_{I_j} \boldsymbol{U}_h^{n+1}\!\cdot\!\boldsymbol{\phi}\,{\mathrm{d}r}
		&= \int_{I_j} \boldsymbol{U}_\sigma^{n}\!\cdot\!\boldsymbol{\phi}\,{\mathrm{d}r}
		+ \frac{\Delta t}{6}\,
		\mathcal{L}_j^{\mathrm{DG}}\!\bigl(\boldsymbol{U}_\sigma^{n,0},A_h^{n,0},B_h^{n,0};\boldsymbol{\phi}\bigr)\\
		&\quad+ \frac{\Delta t}{3}\,
		\mathcal{L}_j^{\mathrm{DG}}\!\bigl(\mathcal{B}\boldsymbol{U}_h^{n,1},A_h^{n,1},B_h^{n,1};\boldsymbol{\phi}\bigr)\\
		&\quad+ \frac{\Delta t}{3}\,
		\mathcal{L}_j^{\mathrm{DG}}\!\bigl(\mathcal{B}\boldsymbol{U}_h^{n,2},A_h^{n,2},B_h^{n,2};\boldsymbol{\phi}\bigr)\\	
		&\quad+ \frac{\Delta t}{6}\,
		\mathcal{L}_j^{\mathrm{DG}}\!\bigl(\mathcal{B}\boldsymbol{U}_h^{n,3},A_h^{n,3},B_h^{n,3};\boldsymbol{\phi}\bigr).
	\end{align*}

	\bibliographystyle{siamplain}
	\bibliography{references}
\end{document}